\documentclass{amsart}
\title[Exotic $t$-structures for two-block Springer fibres]{Exotic $t$-structures for two-block Springer fibres}
\author{Rina Anno}
\email{ranno@mit.edu}
\address{Massachusetts Institute of Technology, Department of Mathematics, 77 Massachusetts Avenue, Cambridge, MA 02139-4307}
\author{Vinoth Nandakumar}
\email{vinoth@math.utah.edu}
\address{University of Utah, Department of Mathematics, 155 S 1400 E, Salt Lake City, Utah, 84102}
\usepackage{amsmath}
\usepackage{amssymb}
\usepackage{amsfonts}
\usepackage{amsthm}
\usepackage{mathrsfs}
\usepackage{graphicx}
\graphicspath{{images/}}
\usepackage[all]{xy}

\newtheorem*{theorem*}{Theorem}
\newtheorem{theorem}{Theorem}[section]
\newtheorem{lemma}[theorem]{Lemma}
\newtheorem{corollary}[theorem]{Corollary}
\newtheorem{proposition}[theorem]{Proposition}

\theoremstyle{definition}
\newtheorem{definition}[theorem]{Definition}
\newtheorem{remark}[theorem]{Remark}
\newtheorem{example}[theorem]{Example}

\begin{document}

\dedicatory{Dedicated to our teacher, Roman Bezrukavnikov}
\begin{abstract} We study the exotic t-structure on $\mathcal{D}_n$, the derived category of coherent sheaves on two-block Springer fibre (i.e. for a nilpotent matrix of type $(m+n,n)$ in type $A$). The exotic t-structure has been defined by Bezrukavnikov and Mirkovic for Springer theoretic varieties in order to study representations of Lie algebras in positive characteristic. Using work of Cautis and Kamnitzer, we construct functors indexed by affine tangles, between categories of coherent sheaves on different two-block Springer fibres (i.e. for different values of $n$). After checking some exactness properties of these functors, we describe the irreducible objects in the heart of the exotic t-structure on $\mathcal{D}_n$ and enumerate them by crossingless $(m,m+2n)$ matchings. We compute the $\text{Ext}$'s between the irreducible objects, and show that the resulting algebras are an annular variant of Khovanov's arc algebras. In subsequent work we will make a link with annular Khovanov homology, and use these results to give a characteristic $p$ analogue of some categorification results using two-block parabolic category $\mathcal{O}$ (by Bernstein-Frenkel-Khovanov, Brundan, Stroppel, et al).  
\end{abstract}

\maketitle

\let\thefootnote\relax

\section{Introduction}

Let $G$ be a semi-simple Lie group, with Lie algebra $\mathfrak{g}$, flag variety $\mathcal{B}$ and nilpotent cone $\mathcal{N}$. It is well-known that there is a natural map $\pi: T^* \mathcal{B} \rightarrow \mathcal{N}$ which is a resolution of singularities (known as the Springer resolution). Given $e \in \mathcal{N}$, let $\mathcal{B}_e = \pi^{-1}(e)$; these varieties are known as Springer fibers, and are of special interest in representation theory. For instance, in type $A$, Springer showed that the top cohomology of a Springer fiber can be equipped with a representation of the Weyl group, and further realizes an irreducible representation. 

This special case when $G=SL(m+2n)$, and the nilpotent $e$ has Jordan type $(m+n, n)$, is easier to understand, and has been studied extensively. In \cite{sw}, Stroppel and Webster study the geometry and combinatorics of these ``two-block Springer fibers" and investigate connections with Khovanov's arc algebras. In \cite {hr}, Russell studies the topology of these varieties, and describes a certain basis in the Springer representation.

In \cite{bm}, Bezrukavnikov and Mirkovic introduce ``exotic t-structures" on derived categories of coherent sheaves on Springer theoretic varieties, in order to study the modular representation theory of $\mathfrak{g}$. These exotic $t$-structures are defined using a certain action of the affine braid group $\mathbb{B}_{aff}$ on these categories, which was defined by Bezrukavnikov and Riche (see \cite{br}). 

More precisely, let $\textbf{k}$ be an algebraically closed field of characteristic $p$ with $p > h$ (here $h$ is the Coxeter number), and let $\mathfrak{g}$ be an arbitrary reductive group defined over $\textbf{k}$. Let $\lambda \in \mathfrak{h}_{\textbf{k}}$ be integral and regular; and let $e \in \mathcal{N}(\textbf{k})$ be a nilpotent. Let $\text{Mod}^{fg, \lambda}_{e}(\text{U}\mathfrak{g})$ be the category of modules with generalized central character $(\lambda, e)$. Theorem $5.3.1$ from \cite{bmr} (see also Section $1.6.2$ from \cite{bm}) states that there is an equivalence: \begin{equation} \label{bmrequiv} D^b(\text{Coh}_{\mathcal{B}_{e, \textbf{k}}}(\widetilde{\mathfrak{g}}_{\textbf{k}})) \simeq D^b(\text{Mod}^{fg, \lambda}_{e}(\text{U} \mathfrak{g}_{\textbf{k}})) \end{equation}

Further, it is proven that the tautological $t$-structure on the derived category of modules, corresponds to the exotic $t$-structure on the derived category of coherent sheaves.

Here we will study exotic $t$-structures for the case of two-block Springer fibers in type $A$ (ie. for a nilpotent of Jordan type $(m+n, n)$), give a description of the irreducible objects in the heart of the t-structure, and the Ext spaces between these irreducibles. 

In \cite{bs}, Brundan and Stroppel show that the principal block of parabolic category $\mathcal{O}^{\mathfrak{p}}$, for the parabolic $\mathfrak{p}$ with Levi $\mathfrak{gl}_m \oplus \mathfrak{gl}_n$ inside $\mathfrak{gl}_{m+n}$, is governed by a diagram algebra that is closely related to Khovanov's arc algebra. Further, work by Bernstein-Frenkel-Khovanov (see \cite{bfk}) and Stroppel (see \cite{s}) shows that Reshetikhin-Turaev invariants for $\mathfrak{sl}_2$, indexed by linear tangles, may be categorified by certain functors between these categories. Using the machinery developed in this paper, in future work we will give a characteristic $p$ analogue of this story; the category $\text{Mod}^{fg, \lambda}_{e}(\text{U}\mathfrak{g})$ can be thought of as a characteristic $p$ analogue of $\mathcal{O}^{\mathfrak{p}}$. While the former construction (in \cite{bfk}) naturally gives rise to Khovanov homology, in the characteristic $p$ setting one will obtain annular Khovanov homology (which was developed by Grigsby, Licata and Wehrli in \cite{glw}). See section $6.2$ and $6.3$ for more details.

Now let us describe the contents of this paper in more detail. 

\subsection{Two-block Springer fibers} In Section 1, we recall the definition and some properties of two-block Springer fibers, and define the categories that we will be studying. Let $m \geq 0$ be fixed; and let $n \in \mathbb{Z}_{\geq 0}$ vary. Consider the Lie algebra $\mathfrak{g} = \mathfrak{sl}_{m+2n}$, and denote the nilpotent cone of $\mathfrak{sl}_{m+2n}$ (the variety consisting of nilpotent matrices of size $m+2n$) by $\mathcal{N}_n$. Denote by $z_n$ the standard nilpotent of type $(m+n, n)$: \begin{align*} z_n &= \left( \begin{array}{cccccccc}
� & 1 & � & � & � & � & � & � \\
� & � & \cdots & � & � & � & � & � \\
� & � & � & 1 & � & � & � & � \\
0 & 0 & \cdots & 0 & 0 & 0 & \cdots  & 0 \\
� & � & � & � & � & 1 & � & � \\
� & � & � & � & � & � & \cdots & � \\
� & � & � & � & � & � & � & 1 \\
0 & 0 & \cdots & 0 & 0 & 0 & \cdots & 0
\end{array} \right) \end{align*}
Let $\mathcal{B}_n$ be the flag variety for $GL_{m+2n}$. The Springer resolution is $T^* \mathcal{B}_n$: 
\begin{align*} \mathcal{B}_n &= \{ (0 \subset V_1 \subset \cdots \subset V_{m+2n}=\mathbb{C}^{m+2n})\ | \text{ dim }V_i = i  \} \\ T^* \mathcal{B}_n &= \{ (0 \subset V_1 \subset \cdots \subset V_{m+2n}=\mathbb{C}^{m+2n}, x)\ |\ x \in \mathfrak{sl}_{m+2n}, xV_i \subset V_{i-1} \} \\ \end{align*} 

The natural projection $\pi_n: T^* \mathcal{B}_n \rightarrow \mathcal{N}_n$ is a resolution of singularities. The two-block Springer fiber is the variety $$\mathcal{B}_{z_n} = \pi_n^{-1}(z_n) = \{ (0 \subset V_1 \subset \cdots \subset V_{m+2n}) \in \mathcal{B}_n \  | \  z_n V_{i} \subseteq V_{i-1} \}$$
The Mirkovic-Vybornov transverse slices $S_n \subset \mathfrak{g}$ is a variant of the Slodowy slice. The following variety is of interest, since it is a resolution of $S_n \cap \mathcal{N}$. 
\begin{align*} U_n &= \pi_n^{-1}(S_n) \subset T^* \mathcal{B}_n \\ &= \{ (0 \subset V_1 \subset \cdots \subset V_{m+2n}=\mathbb{C}^{m+2n}, x)\ |\ x \in S_n, xV_i \subset V_{i-1} \} \end{align*}
Let $\mathcal{D}_n = D^b(\text{Coh}_{\mathcal{B}_{z_n}}(U_n))$ be the bounded derived category of coherent sheaves on $U_n$, which are supported on $\mathcal{B}_{z_n}$. These are the categories that we will be studying.

\subsection{Affine tangles} In Section $2$, we recall the definition, and some properties, of affine tangles. 

\textbf{Definition.} Let $p, q$ be positive integers of the same parity. A $(p,q)$ affine tangle is an embedding of $\frac{p+q}{2}$ arcs and a finite number of circles into the region $\{ (x,y) \in \mathbb{C} \times \mathbb{R} | 1 \leq |x| \leq 2 \}$, such that the end-points of the arcs are $(1,0),(\zeta_p, 0), \cdots , (\zeta_p^{p-1}, 0), (2, 0), (2\zeta_q,0), \cdots, (2 \zeta_q^{q-1},0)$ in some order (where $\zeta_k = e^{\frac{2\pi i}{k}}$).

\textbf{Definition.} Let \textbf{ATan} be the category with objects $\{ k \}$ for $k \in \mathbb{Z}_{\geq 0}$, and the morphisms between $p$ and $q$ consist of all affine $(p, q)$ tangles (up to isotopy). 

The above definition is consistent, since a $(p,q)$ affine tangle $\alpha$, and a $(q,r)$ affine tangle $\beta$, then $\beta \circ \alpha$ is a $(p,r)$ affine tangle. 

We recall the well-known presentation of this category using generators and relations. The generators consist of ``cups", $g_n^i$, which are $(n-2, n)$ tangles; ``caps", $f_n^i$, which are $(n, n-2)$ tangles, ``crossings", $t_n^i(1), t_n^i(2)$ and rotations $r_n$, $r_n'$, which are $(n,n)$ tangles. The relations are listed in Definition \ref{DefinitionATan}. In this paper we work with the category $\textbf{AFTan}$ of affine framed tangles that has additional generators $w_n^i(1)$ and $w_n^i(2)$ that twist the framing of the $i$th strand.

\subsection{Functors associated to affine tangles} \quad

\textbf{Definition.} Let $\textbf{AFTan}_m$ be the full subcategory of \textbf{AFTan}, containing the objects $\{ m+2n \}$ for $n \in \mathbb{Z}_{\geq 0}$. A ``weak representation" of the category $\textbf{AFTan}_m$ is an assignment of a triangulated category $\mathcal{C}_n$ for each $n \in \mathbb{Z}_{\geq 0}$, and a functor $\Psi(\alpha): \mathcal{D}_p \rightarrow \mathcal{D}_q$ for each affine framed $(m+2p, m+2q)$ tangle $\alpha$, such that the relations between tangles hold for these functors: i.e. if $\beta$ is an $(m+2q, m+2r)$-tangle, then there is an isomorphism $\Psi(\beta) \circ \Psi(\alpha) \simeq \Psi(\beta \circ \alpha)$. 

The main result of this section is a construction of a weak representation of $\textbf{AFTan}_m$ using the categories $\mathcal{D}_n$ above. To do this, we mimic the strategy used by Cautis and Kamnitzer in \cite{ck}, where they construct a weak representation of the category $\textbf{OTan}$ of oriented (non-affine) tangles, using slightly larger categories. 

\subsection{The exotic $t$-structure on $\mathcal{D}_n$} In Section $5$, we recall the definition of exotic $t$-structures (introduced by Bezrukavnikov and Mirkovic in \cite{bm}), and describe how they are related to the action of affine tangles constructed above. 

Let $\mathbb{B}_{aff}$ be the affine braid group. As a special case of the construction in Section $1$ of \cite{bm} (see also Bezrukavnikov-Riche, \cite{br}), we have an action of $\mathbb{B}_{aff}$ on $\mathcal{D}_n$ (ie. for every $b \in \mathbb{B}_{aff}$, there exists a functor $\Psi(b): \mathcal{D}_n \rightarrow \mathcal{D}_n$, and an isomorphism $\Psi(b_1 b_2) \simeq \Psi(b_1) \circ \Psi(b_2)$ for $b_1, b_2 \in \mathbb{B}_{aff}$). It turns out that $\mathbb{B}_{aff}$ can be identified as a subgroup of the monoid of $(m+2n, m+2n)$-tangles; and under this identification, the action of $\mathbb{B}_{aff}$ coincides with the action constructed above. 

Let $\mathbb{B}_{aff}^+ \subset \mathbb{B}_{aff}$ be the semigroup generated by the lifts of the simple reflections $\tilde{s}_{\alpha}$ in the Coxeter group $W_{aff}^{Cox}$. Bezrukavnikov-Mirkovic's construction in \cite{bm} specializes to give an exotic $t$-structure on $\mathcal{D}_n$, which is defined as follows: 
\begin{align*} \mathcal{D}_n^{\geq 0} &= \{ \mathcal{F} \  | \  R \Gamma(\Psi(b^{-1})\mathcal{F}) \in D^{\geq 0}(\text{Vect} ) \  \forall \  b \in \mathbb{B}_{aff}^+ \} \\ \mathcal{D}_n^{\leq 0} &= \{ \mathcal{F} \  | \  R \Gamma(\Psi(b) \mathcal{F}) \in D^{\leq 0}(\text{Vect} ) \  \forall \  b \in \mathbb{B}_{aff}^+ \} \\ \end{align*} We also prove that the ``cup" functors $\Psi(g_{n}^i)$ are exact with the exotic $t$-structures, and send irreducible objects to irreducible objects (Theorem \ref{irred}).

\subsection{Irreducible objects in the heart of the exotic t-structure on $\mathcal{D}_n$} In Section $6$, we give a description of the irreducible objects in the exotic $t$-structure on $\mathcal{D}_n$, and compute the \text{Ext} spaces between them. 

Let $\text{Cross}(m,n)$ be the set of affine $(m, m+2n)$ tangles, where the $m$ inner points are not labelled, the $m+2n$ outer points are labelled, and whose vertical projections to $\mathbb{C}$ do not have crossings. For every $\alpha\in Cross(n)$ we have a functor $\Psi(\alpha): \mathcal{D}_0 \rightarrow \mathcal{D}_n$; let $\Psi_{\alpha} = \Psi(\alpha) (\underline{\mathbb{C}})$ (here $ \underline{\mathbb{C}} \in D^b(\text{Vect}) \simeq \mathcal{D}_0$). We show that  that $\{ \Psi_{\alpha} \, | \, \alpha \in \text{Cross}(m,n) \}$ constitute the irreducible objects in $\mathcal{D}_n^0$ (Proposition \ref{PropIrredObjects}). 

We also prove that for $\beta \in \text{Cross}(m,n)$, $\text{Ext}^{\bullet}(\Psi_{\alpha}, \Psi_{\beta})$ is given by the below formula. Here $\Lambda$ denotes a complex in $D^b(\text{Vect})$ concentrated in degrees $1$ and $-1$; and $\check{\alpha}$ is the $(m+2n,m)$ affine tangle obtained by ``inverting" $\alpha$. An a $(m,m)$ affine tangle $\gamma$ with no crossings is said to be ``good" if it has no cups or caps, and $\omega(\gamma)$ denote the number of circles present. In Theorem \ref{final}, we prove that

\begin{align*} \text{Ext}^{\bullet}(\Psi_{\alpha}, \Psi_{\beta}) = \begin{cases} \Lambda^{\otimes \omega(\check{\alpha} \circ \beta)}[-n] \; &\mbox{if   }\check{\alpha} \circ \beta \text{ is good}  \\ 0 \; &\mbox{otherwise  } \end{cases} \end{align*}

We also give a conjectural description of the multiplication in the algebra $$\text{Ext}^{\bullet}(\bigoplus_{\alpha \in \text{Cross}(m,n)} \Psi_{\alpha})$$ 

\subsection{Further directions} In the equivalence (\ref{bmrequiv}), the heart of the exotic t-structure is identified with an abelian category of modules over $U \mathfrak{g}$ having a fixed central character. Thus the simple objects that we have classified in the heart of the exotic $t$-structure will correspond to irreducible representations with that fixed central character. In future work, we plan to study these modules (e.g. compute dimensions, and give character formulaes) by using our description of these exotic sheaves. 

Using techniques developed by Cautis and Kamnitzer, we can show the Grothendieck group of the category $\mathcal{D}_n$ can be naturally identified with $V^{\otimes m+2n}_{[m]}$, the $m$-weight space in $V^{\otimes m+2n}$ (here $V = \mathbb{C}^2$, considered as an $\mathfrak{sl}_2$ representation). By looking at the images of the functors $\Psi(\alpha)$ in the Grothendieck group, we obtain a map $$ \hat{\psi}: \{ (m+2k, m+2l) \text{-affine tangles} \} \rightarrow \text{Hom}(V^{\otimes m+2k}_{[m]}, V^{\otimes m+2l}_{[m]})$$ We expect that this map will coincide with a well-known invariant for affine tangles, and that the images of the irreducible objects $\Psi_{\alpha}$ in the Grothendieck group will be the canonical basis (or perhaps the dual canonical basis). Inspired by Khovanov's construction in \cite{khov1} and \cite{khov2}, we also expect that it will be possible to give an alternate categorification of $\hat{\psi}$, using categories of modules over the Ext algebras controlling $\mathcal{D}_n$ (which closely resemble Khovanov's arc algebras). 

\subsection{Acknowledgements} We would like to thank Roman Bezrukavnikov, for suggesting this project to us, and for numerous helpful discussions and insights. We would also like to thank Paul Seidel for suggesting the study of $m=0$ case to the first author a while ago. We are also grateful to Joel Kamnitzer, Mikhail Khovanov, Catherina Stroppel, Ben Webster and David Yang for many helpful discussions; and to Anthony Henderson for help with the proof of Lemma \ref{lemma-transverse}. The first author would like to thank the University of Pittsburgh and the second author would like to thank the University of Sydney, where part of this work was completed. 

\section{Two-block Springer fibres}\label{section-varieties}

\subsection{Transverse slices for two-block nilpotents}


Fix $m \geq 0$. For $n \in \mathbb{Z}_{\geq 0}$, let $z_n$ be the standard nilpotent of Jordan type $(m+n,n)$. Let $S_n\subset \mathfrak{sl}_{m+2n}$ denote the Mirkovic-Vybornov transverse slice to the nilpotent $z_n$ (see section $3.3.1$ in \cite{mv}): 
\begin{align*} S_n = \{ z_n + \sum_{1 \leq i \leq m+2n} a_{i} e_{m+n,i} + \sum_{i \in \{1, \cdots, n, m+n+1, \cdots, m+2n \}} b_i e_{m+2n, i} \} \end{align*}
\begin{definition} Denote by $\mathcal{N}_n$ the nilpotent cone for $\mathfrak{sl}_{m+2n}$. Let $\mathcal{B}_n$ denote the complete flag variety for $GL_{m+2n}(\mathbb{C})$, and for $0 < k < m+2n$ define the varieties $\mathcal{P}_{k,n}$ as follows:
\begin{equation*} \mathcal{P}_{k,n}=\{ (0 \subset V_1 \subset \cdots \subset \widehat{V_k} \subset \cdots \subset V_{m+2n}=\mathbb{C}^{m+2n}) \}. 
\end{equation*}
Then the varieties $T^* \mathcal{B}_n$, $T^* \mathcal{P}_{k,n}$ can be described as follows:
\begin{align*}
T^* \mathcal{B}_n &= \{ (0 \subset V_1 \subset \cdots \subset V_{m+2n}=\mathbb{C}^{m+2n}, x)\ |\ x \in \mathfrak{sl}_{m+2n},\ xV_i \subset V_{i-1} \}; \\ 
T^* \mathcal{P}_{k,n}& = \{ (0 \subset V_1 \subset \cdots \subset \widehat{V_k} \subset \cdots \subset V_{m+2n}=\mathbb{C}^{m+2n}),x\ |\ \\ &x \in \mathfrak{sl}_{m+2n},\ xV_{k+1} \subset V_{k-1},\ xV_i \subset V_{i-1} \text{ for } i \neq k,k+1 \}. \end{align*} \end{definition}

Pick a basis $e_1, \ldots, e_{m+n+1}, f_1, \ldots, f_{n+1}$ of $\mathbb{C}^{m+2n+2}$ so that $z_{n+1} e_{i}= e_{i-1}$, $z_{n+1} f_j = f_{j-1}$ (where we set $e_0 = f_0 = 0$).

\begin{lemma} \label{1} For any $x \in S_{n+1}$ such that $\text{dim}(\text{Ker } x)=2$, we have $\text{Ker } x = \mathbb{C}e_1 \oplus \mathbb{C}f_1$, and there is a natural isomorphism $\phi_x: x V_{m+2n+2} \simeq \mathbb{C}^{m+2n}$. \end{lemma} \begin{proof} By the construction in \cite[section 3.3.1]{mv} we can assume that $xe_i = e_{i-1} + a_i e_{m+n+1} + c_i f_{m+1}$ if $i \leq m+1$, $xe_i = e_{i-1} + a_i e_{m+n+1}$ if $i > m+1$, and $x f_{j} = f_{j-1} + b_j e_{m+n+1} + d_j f_{m+1}$. Then we have:
\begin{align*} x\left(\sum_{1 \leq i \leq m+n+1} \lambda_i e_i + \sum_{1 \leq j \leq m+1} \nu_j f_j\right) &= \sum_{1 \leq i \leq m+n} \lambda_{i+1} e_i + \left(\sum_{1 \leq i \leq m+n+1} a_i \lambda_i + \sum_{1 \leq j \leq m+1} b_j \nu_j\right)e_{m+n+1} \\+&\sum_{1 \leq j \leq m} \nu_{j+1} f_j + \left(\sum_{1 \leq i \leq m+1} a_i \lambda_i+ \sum_{1 \leq j \leq m+1} d_j \nu_j\right)f_{m+1} \label{eqn1} \\ \end{align*} So $xv=x\bigl(\sum_{1 \leq i \leq m+n+1} \lambda_i e_i + \sum_{1 \leq j \leq m+1} \nu_j f_j\bigr)=0$ implies that $\lambda_i = \nu_j = 0$ for $i,j>1$, i.e. that $v \in \mathbb{C}e_1 \oplus \mathbb{C}f_1$. If $xv=0$ it follows that $a_1=b_1=c_1=d_1=0$. So:
\begin{multline*} 
x V_{m+2n+2} = \biggl\{
\sum_{1 \leq i \leq m+n} \lambda_i e_i + 
\left(\sum_{1 \leq i \leq m+n} a_{i+1} \lambda_i + \sum_{1 \leq j \leq n} b_{j+1} \nu_j\right)e_{m+n+1} + \\ 
\sum_{1 \leq j \leq n} \mu_j f_j + \left(\sum_{1 \leq i \leq m} c_{i+1} \lambda_i + \sum_{1 \leq j \leq n} d_{j+1} \nu_j\right) f_{n+1} 
\biggr\}
\end{multline*}
Let $\gamma_{m,n}: \mathbb{C}^{m+2n+2}=(\bigoplus_{1 \leq i \leq m+n} \mathbb{C}e_i \oplus \bigoplus_{1 \leq j \leq n} \mathbb{C} f_j) \oplus (\mathbb{C} e_{m+n+1} \oplus \mathbb{C} f_{n+1}) \rightarrow (\bigoplus_{1 \leq i \leq m+n} \mathbb{C}e_i \oplus \bigoplus_{1 \leq j \leq n} \mathbb{C} f_j)$ denote the natural projection map. Now $\phi_x := \gamma_{m,n}|_{x V_{m+2n+2}}: xV_{m+2n+2} \rightarrow \bigoplus_{1 \leq i \leq m+n} \mathbb{C}e_i \oplus \bigoplus_{1 \leq j \leq n} \mathbb{C} f_j$ is an isomorphism.
\end{proof}

\begin{proposition}\label{prop-isomorphism-Un-partial_Un+1}
 For every $0<k<m+2n+2$ we have an isomorphism of varieties $S_{n+1} \times_{\mathfrak{sl}_{m+2n+2}} T^*\mathcal{P}_{k,n+1} \simeq S_n \times_{\mathfrak{sl}_{m+2n}} T^*\mathcal{B}_{n}$. 
\end{proposition} 
\begin{proof} 
By definition: 
\begin{align*} 
S_{n+1} \times_{\mathfrak{sl}_{m+2n+2}} T^*\mathcal{P}_{k,n+1} = &
\{ (0 \subset V_1 \subset \cdots \subset \widehat{V_k} \subset \cdots \subset V_{m+2n+2},\ x) \  | \\
 & \  x \in S_{n+1},\ xV_{k+1} \subset V_{k-1},\ xV_i \subset V_{i-1} \text{ for } i \neq k,k+1 \}; 
\end{align*} 
\begin{align*} S_n \times_{\mathfrak{sl}_{m+2n}} T^*\mathcal{B}_{n} = \{ (0 \subset W_1 \subset \cdots \subset W_{m+2n} = \mathbb{C}^{m+2n}, \ y)\  |\  y \in S_n, \ y W_{i} \subset W_{i-1} \}. 
\end{align*} 
Since $x \in S_{n+1}$, the Jordan type of $x$ is a two-block partition, and $\text{dim}(\text{Ker}(x)) \leq 2$; but $x V_{k+1} \subset V_{k-1}$ so we must have $x V_{k+1} = V_{k-1}$. Consider the flag $(0 \subset V_1 \subset \cdots \subset V_{k-1} = x V_{k+1} \subset x V_{k+2} \subset \cdots \subset x V_{m+2n+2})$. 
Recall the isomorphism $\phi_x: x V_{m+2n+2} \xrightarrow{\sim}\mathbb{C}^{m+2n}$ from Lemma \ref{1} and denote by $\Phi(x) \in \text{End}(\mathbb{C}^{m+2n})$ the endomorphism induced on $\mathbb{C}^{m+2n}$ by the action of $x$ on $x V_{m+2n+2}$. Construct a map $\alpha: S_{n+1} \times_{\mathfrak{sl}_{m+2n+2}} T^*\mathcal{P}_{k,n+1} \rightarrow T^*\mathcal{B}_{n}$ as follows: 
\begin{multline*} \alpha(0 \subset V_1 \subset \cdots \subset V_{m+2n}, x)= \\
=((0 \subset \phi_x(V_1) \subset \cdots \subset \phi_x(V_{k-1}) = \phi_x(x V_{k+1}) \subset \phi_x(x V_{k+2}) \subset \cdots \subset \mathbb{C}^{m+2n}) ,\Phi(x)) 
\end{multline*}

We claim that $\alpha$ gives the required isomorphism $S_{n+1} \times_{\mathfrak{sl}_{m+2n+2}} T^*\mathcal{P}_{k,n+1} \simeq S_n \times_{\mathfrak{sl}_{m+2n}} T^*\mathcal{B}_{n}$. First we check that $\Phi(x) \in S_n$. From the argument in Lemma \ref{1}, $\Phi(x) e_{i} = e_{i-1} + a_{i+1} e_{m+n} + c_{i+1} f_n$ if $i \leq n$, $\Phi(x) e_{i} = e_{i-1} + a_{i+1} e_{m+n}$ if $i > n$, and $\Phi(x) f_{j} = f_{j-1} + c_{j+1} e_{m+n} + d_{j+1} f_n$. Thus $\Phi$ gives a bijection between $\{ x \in S_{n+1} \cap \mathcal{N}_{n+1} | \text{ dim}(\text{Ker }x) = 2 \}$ and $S_{n} \cap \mathcal{N}_n$. It follows that $\alpha$ has image $S_n \times_{\mathfrak{sl}_{m+2n}} T^*\mathcal{B}_{n}$ and that $\alpha$ is an isomorphism onto its image, as required. \end{proof}

Let us define the varieties and categories that we are going to use throughout the paper.

\begin{definition}\label{definition-Un-Xni}
 Under the Springer resolution map $\pi_n: T^* \mathcal{B}_n \rightarrow \mathcal{N}_{n}$, let $\mathcal{B}_{z_n} = \pi_n^{-1}(z_n)$. 
Let 
\begin{align*}
U_n = S_n \times_{\mathfrak{sl}_{m+2n}} T^*\mathcal{B}_{n} =
& \{ (0 \subset V_1 \subset \cdots \subset V_{m+2n},\ x) \ | \ x \in S_{n},\ xV_j \subset V_{j-1}\  \forall j \}.
\end{align*}
\begin{align*} 
X_{n,i} = S_n \times_{\mathfrak{sl}_{m+2n}} T^*\mathcal{P}_{i,n} \times_{\mathcal{P}_{i,n}} \mathcal{B}_n =
& \{ (0 \subset V_1 \subset \cdots \subset V_{m+2n},\ x) \ | \\ 
&\ x \in S_{n},\ xV_{i+1} \subset V_{i-1},\ xV_j \subset V_{j-1}\  \forall j \}.
\end{align*}
Define $\mathcal{D}_n =D^b(\text{Coh}_{\mathcal{B}_{z_n}}(U_n))$ to be the bounded derived category of coherent sheaves on $U_n$ supported on $\mathcal{B}_{z_n}$. 
\end{definition}

Note that $X_{n,i}$ is a closed subvariety of $U_n$ of codimension $1$. On the other hand, the projection of $X_{n,i}$ onto
$S_n \times_{\mathfrak{sl}_{m+2n}} T^*\mathcal{P}_{i,n}$, which is by Proposition \ref{prop-isomorphism-Un-partial_Un+1} 
isomorphic to $U_{n-1}$, is a $\mathbb{P}^1$-bundle. Indeed, the fiber over each point
$ (0 \subset V_1 \subset \cdots \subset \widehat{V_k} \subset \cdots \subset V_{m+2n+2}, \ x)$ is isomorphic to 
$\mathbb{P}(V_{i+1}/V_{i-1})$.

\subsection{Description of the general setup}\label{subsection-varieties-general-setup}


Our geometric setup is going to be be similar to that of \cite{ck}, so we will describe both alongside and point out the dependencies and the differences.
Consider a $2(m+2n)$-dimensional vector space $V_{m,n}$ with basis $e_1, \ldots , e_{m+2n}, f_1, \ldots ,f_{m+2n}$ and a nilpotent $z$ such that $z e_{i} = e_{i-1}$, $z f_{i} = f_{i-1}$. Let $W_{m,n} \subset V_{m,n}$ denote the vector subspace with basis $e_1, \ldots, e_{m+n}, f_1, \ldots, f_n$, so that $z|_{W_{m,n}}$ has Jordan type $(m+n,n)$; we will identify $W_{m,n}$ with $V_{m+2n}$. Let $P:V_{m,n} \rightarrow W_{m,n}$ denote the projection defined by $Pe_{i} = e_{i}$ if $i \leq m+n$, $Pe_{i}=0$ if $i > m+n$; $Pf_i = f_i$ if $i \leq n$, $Pf_i = f_i$ if $i > n$.
In Section $2$ of \cite{ck}, the following four series of varieties are defined (for $m=0$):
\begin{align*}
Y_{m+2n} & = \{ (L_1 \subset \cdots \subset L_{m+2n} \subset V_{m,n}) | \text{ dim}\  L_i = i,\ zL_{i} \subset L_{i-1} \};\\
Q_{m+2n} & =\{ (L_1 \subset \cdots \subset L_{m+2n}) \in Y_{m+2n} | P(L_{m+2n}) = W_{m,n} \};\\
X_{m+2n}^i & = \{ (L_1 \subset L_2 \subset \cdots \subset L_{m+2n}) | \  L_{i+1}=z^{-1}(L_{i-1}) \};\\
Z_{m+2n}^i & = \{ (L,L') \in Y_{m+2n} \times Y_{m+2n} | \  L_j = L'_{j} \  \forall \  j \neq i \}.
\end{align*}
In the notation of \cite{ck}, the variety $Q_{m+2n}$ should be denoted by $U_{m+2n}$, but we chose to call it $Q_{m+2n}$ here to
avoid the confusion with our $U_n$.
The relationships between these varieties are as follows: $Q_{m+2n}\subset Y_{m+2n}$ is an open subset, and  $X_{m+2n}^i\subset Y_{m+2n}$ is a closed
subset. Moreover, $X_{m+2n}^i$ is fibered over $Y_{m+2n-2}$ with fiber $\mathbb{P}^1$, and thus can be considered a closed subset in
$Y_{m+2n}\times Y_{m+2n-2}$. 

Cautis and Kamnitzer use the categories $D(Y_{m+2n})$ for their categorification, and utilize the varieties
$X_{m+2n}^i\subset Y_{m+2n}\times Y_{m+2n-2}$ and $Z_{m+2n}^i \subset Y_{m+2n} \times Y_{m+2n}$ to construct Fourier-Mukai functors
that generate the tangle category action. We are going to use the varieties $U_n$ and $X_{n,i}\subset U_n\times U_{n-1}$ from Definition
\ref{definition-Un-Xni}
that have similar properties, namely $X_{n,i}\to U_n$ is a closed embedding, and $X_{n,i}\to U_{n-1}$ is a $\mathbb{P}^1$ bundle.
We are going to use the categories $\mathcal{D}_n=D^b(\text{Coh}_{\mathcal{B}_{z_n}}(U_n))$ for the categorification,
and the varieties $X_{n,i}$ will provide the cup and cap tangle generators (see Section \ref{section-tangles} for the description of the tangle category).
While it is true that there is an embedding $U_n\hookrightarrow Y_{m+2n}$ (see section \ref{section-varieties-Un} below) 
and under this embedding we can identify $X_{n,i}\simeq U_n\cap X_{m+2n}^i$, certain geometric facts such as Lemma \ref{transverse} below
do not follow directly from their counterparts in \cite{ck}.
 We will not need the analogue of $Z_{m+2n}^i$ to describe the crossing generators since our Theorem \ref{spherical} allows us 
to define the crossing generators and
prove most tangle relations without direct computations with Fourier-Mukai kernels. 

\subsection{The varieties $U_n$}\label{section-varieties-Un}

We are going to show that $U_n$ is a isomorphic to a closed subvariety of $Q_{m+2n}$ and thus a locally closed subvariety of $Y_{m+2n}$.
To do this, we recall that $U_n \simeq S_n \times_{\mathfrak{sl}_{m+2n}} T^* \mathcal{B}_n$ and present $Q_{m+2n}$ in a similar form.
\begin{definition} \begin{align*} S_n' &= \{ \left( \begin{array}{cccccccc}
� & 1 & � & � & � & � & � & � \\
� & � & \ddots & � & � & � & � & � \\
� & � & � & 1 & � & � & � & � \\
a_1 & a_2 & \cdots & a_{m+n} & b_1 & b_2 & \cdots  & b_{n} \\
� & � & � & � & � & 1 & � & � \\
� & � & � & � & � & � & \ddots & � \\
� & � & � & � & � & � & � & 1 \\
c_1 & c_2 & \cdots & c_{m+n} & d_1 & d_2 & \cdots & d_n
\end{array} \right) \}. \end{align*} \end{definition}
Note that we have $S_n \subset S_n'$, and $U_n \subset S_n' \times_{\mathfrak{sl}_{m+2n}} T^* \mathcal{B}_n$. 
 Now we can prove the following lemma:
\begin{lemma} \label{lin} Given $x \in S_n' \cap \mathcal{N}_n$, there exists a unique subspace $L_{m+2n} \subset V_{m,n}$, with $P L_{m+2n} = W_{m,n}$, such that $z L_{m+2n} \subset L_{m+2n}$ and $PzP^{-1}=x$. \end{lemma} \begin{proof} Since $P L_{m+2n} = W_{m,n}$, to specify the subspace $L_{m+2n}$ it suffices to specify \begin{align*} \tilde{e}_i:=P^{-1}(e_i) = e_{i} + \sum_{1 \leq k \leq n} a_i^{(k)} e_{m+n+k} + \sum_{1 \leq l \leq m+n} c_i^{(l)}f_{n+l} \\ \tilde{f}_j:=P^{-1}(f_j) = f_j + \sum_{1 \leq k \leq n} b_j^{(k)} e_{m+n+k} + \sum_{1 \leq l \leq m+n} d_j^{(l)} f_{n+l} \end{align*} Suppose for $1 \leq i \leq m+n, 1 \leq j \leq n$, $x e_{i} = e_{i-1} + a_i e_{m+n} + c_i f_{n}, xf_j = f_{j-1} + b_j e_{m+n} + d_j f_{n}$; then the identity $PzP^{-1}=x$ is equivalent to $a_i^{(1)}=a_i, c_i^{(1)} = c_i, b_j^{(1)}=b_j$ and $d_j^{(1)}=d_j$. The statement $z L_{m+2n} \subset L_{m+2n}$, i.e. $z \tilde{e}_i, z \tilde{f}_j \in L_{m+2n}$, is equivalent to saying that: \begin{align*} z \tilde{e}_i = \tilde{e}_{i-1} + a_i \tilde{e}_{m+n} + c_i \tilde{f}_n \\ z \tilde{f}_j = \tilde{f}_{j-1} + b_j \tilde{e}_{m+n} + d_j \tilde{f}_n \end{align*} 
Expanding the above two equations: 
\footnotesize
\begin{multline*}
 e_{i-1} + \sum_{1 \leq k \leq n} a_i^{(k)} e_{m+n+k-1} + \sum_{1 \leq l \leq m+n} c_i^{(l)}f_{n+l-1} = e_{i-1} + \sum_{1 \leq k \leq n} a_{i-1}^{(k)} e_{m+n+k} + \sum_{1 \leq l \leq m+n} c_{i-1}^{(l)}f_{n+l} +\\
 + a_i\left(e_{m+n} + \sum_{1 \leq k \leq n} a_{m+n}^{(k)} e_{m+n+k} + \sum_{1 \leq l \leq m+n} c_{m+n}^{(l)}f_{n+l}\right)
 + c_i\left(f_n + \sum_{1 \leq k \leq n} b_n^{(k)} e_{m+n+k} + \sum_{1 \leq l \leq m+n} d_n^{(l)} f_{n+l}\right);
\end{multline*}
\begin{multline*}
f_{j-1} + \sum_{1 \leq k \leq n} b_j^{(k)} e_{m+n+k-1} + \sum_{1 \leq l \leq m+n} d_j^{(l)} f_{n+l-1} = f_{j-1} + \sum_{1 \leq k \leq n} b_{j-1}^{(k)} e_{m+n+k} + \sum_{1  \leq l \leq m+n} d_{j-1}^{(l)} f_{n+l}+ \\
 + b_j\left(e_{m+n} + \sum_{1 \leq k \leq n} a_{m+n}^{(k)} e_{m+n+k} + \sum_{1 \leq l \leq m+n} c_{m+n}^{(l)}f_{n+l}\right)
 + d_j\left(f_n + \sum_{1 \leq k \leq n} b_n^{(k)} e_{m+n+k} + \sum_{1 \leq l \leq m+n} d_n^{(l)} f_{n+l}\right).
\end{multline*} 
\normalsize
Extracting coefficients of $e_{m+n+k}$ and $f_{n+l}$ in the above two equations gives: \begin{align*} a_i^{(k+1)} &= a_{i-1}^{(k)} + a_i a_{m+n}^{(k)} + c_i b_n^{(k)}, \qquad b_j^{(k+1)} = b_{j-1}^{(k)} + b_j a_{m+n}^{(k)} + d_j b_n^{(k)} \\ c_i^{(l+1)} &= c_i^{(l)} + a_i c_{m+n}^{(l)} + c_i d_n^{(l)}, \qquad d_j^{(l+1)} = d_{j-1}^{(l)} + b_j c_{m+n}^{(l)} + d_j d_n^{(l)} \end{align*} Consider the matrix coefficients $(x^k)_{p,q}$ for $1 \leq p,q \leq m+2n$. It follows by induction that we have $a_i^{(k)}= (x^k)_{m+n,i}, b_j^{(k)}= (x^k)_{m+n,m+n+j}, c_i^{(l)}= (x^l)_{m+2n,i}, d_j^{(l)}=(x^l)_{m+2n,m+n+j}$. Indeed, the case where $k=l=1$ is clear; and the induction step follows from expanding the equation $(x^{r+1})_{uv} = \sum_{1 \leq w \leq m+2n} (x^r)_{uw} (x)_{wv}$ for $u=m+n$ and $u=m+2n$.

Using the above recursive definition of $a_{i}^{(k)}$, $b_{j}^{(k)}$, $c_{i}^{(l)}$, and $d_{j}^{(l)}$, it remains to prove that $a_{i}^{(n+1)}=b_{j}^{(n+1)}=0$ and $c_{i}^{(m+n+1)}=d_{j}^{(m+n+1)}=0$. Thus we must show that $(x^{n+1})_{m+n,p}=(x^{m+n+1})_{m+2n,p}=0$ given $1 \leq p \leq m+2n$. Using the equation $(x^{r+1})_{uv} = \sum_{1 \leq w \leq m+2n} (x)_{uw} (x^r)_{wv}$, we compute that: \begin{align*} (x^{n+1})_{m+n,p}&=(x^{n+2})_{m+n-1,p}= \cdots = (x^{m+2n})_{1,p}=0 \\ (x^{m+n+1})_{m+2n,p}&=(x^{m+n+2})_{m+2n-1,p}= \cdots = (x^{m+2n})_{m+n+1,p}=0 \end{align*} This completes the proof of the existence and uniqueness of a $z$-stable subspace $L_{m+2n} \subset V_{m,n}$ with $P L_{m+2n} = W_{m,n}$ and $PzP^{-1} = x$. \end{proof}

Now we can prove the following generalization of Proposition $2.4$ in \cite{ck}:
\begin{lemma} \label{sub} 
There is an isomorphism $Q_{m+2n} \simeq S_n' \times_{\mathfrak{sl}_{m+2n}} T^* \mathcal{B}_n$. 
\end{lemma} 
\begin{proof} 
Given $(L_1 \subset \cdots \subset L_{m+2n}) \in Q_{m+2n}$, since $P: L_{m+2n} \rightarrow W_{m,n}$ is an isomorphism, we have a nilpotent endomorphism $x = PzP^{-1} \in \text{End}(V_{m+2n})$ (here we identify $W_{m,n}$ and $V_{m+2n}$). If $P^{-1}e_i = e_{i} + v'$, where $v'$ lies in the span of $e_{m+n+1}, \cdots, e_{m+2n}, f_{n+1}, \cdots, f_{m+2n}$, then $zP^{-1}e_i=e_{i-1}+v''$ where $v''$ is in the span of $e_{m+n}, \cdots, e_{m+2n-1}, f_n, \cdots, f_{m+2n-1}$. Hence $PzP^{-1}e_i = x e_{i} \in \text{span}(e_{i-1}, e_{m+n}, f_n)$, and similarly $x f_{i} \in \text{span}(f_{i-1}, e_{m+n}, f_n)$; so $x \in S_n'$. Thus we have a map $\alpha: Q_{m+2n} \rightarrow S_n' \times_{\mathfrak{sl}_{m+2n}} T^* \mathcal{B}_n$ given by $\alpha(L_1, \cdots, L_{m+2n}) = (PzP^{-1}, (P(L_1), P(L_2), \cdots, P(L_{m+2n})))$.

For the converse direction, from the below Lemma \ref{lin} we know that given $x \in S_n' \cap \mathcal{N}_n$ there exists a unique $z$-stable subspace $L_{m+2n} \subset V_{m,n}$ such that $P L_{m+2n} = W_{m,n}$ and $PzP^{-1} = x$; call this subspace $L_{m+2n} = \Theta(x)$. We have an isomorphism $P: \Theta(x) \simeq W_{m,n}$. Thus given an element $((0 \subset V_1 \subset \cdots \subset V_{m+2n}), x) \in S_n' \times_{\mathfrak{sl}_{m+2n}} T^* \mathcal{B}_n$, let $\beta(x)=(0 \subset P^{-1}V_1 \subset P^{-1}V_2 \subset \cdots \subset \Theta_x)$. It is clear that $\alpha$ and $\beta$ are inverse to one another. 
\end{proof}


\subsection{The varieties $X_{n,i}$}

We have a $\mathbb{P}^1$-bundle 
\begin{align*}
\pi_{n,i}: X_{n,i} \rightarrow S_n \times_{\mathfrak{sl}_{m+2n}} T^*\mathcal{P}_{i,n} \simeq S_{n-1} \times_{\mathfrak{sl}_{m+2n-2}} T^*\mathcal{B}_{n-1} = U_{n-1},
\end{align*}
 and the embedding of the divisor $j_{n,i}: X_{n,i} \rightarrow S_n \times_{\mathfrak{sl}_{m+2n}} T^*\mathcal{B}_{n} = U_n$. Thus we can view $X_{n,i}$ as a subvariety of $U_{n-1} \times U_n$.

\begin{lemma} \label{lemma-transverse} 
 For $i \neq j$, the varieties $X_{n,i}$ and $X_{n,j}$ intersect transversely inside $U_n$. 
\end{lemma} 
\begin{proof} 
We will view $U_n$ (and also $X_{n,i}$ and $X_{n,j}$) as a subvariety of $G \times_B \mathfrak{n}$, and compute tangent spaces to $X_{n,i}$ and $X_{n,j}$ at points in $X_{n,i} \cap X_{n,j}$ to show transversality. 

Given $(g, x) \in G \times_B \mathfrak{n}$; first we will calculate the tangent space $T_{(g,x)} (G \times_B \mathfrak{n})$. 
Given $X_1 \in \mathfrak{g}, X_2 \in \mathfrak{n}$, a curve through $(g, x)$ in $G \times \mathfrak{n}$ with tangent direction 
$(g \cdot X_1, X_2)$ is $(g \cdot \text{exp}(\epsilon X_1), x + \epsilon X_2)$. 
Infinitesimally, $(g \cdot \text{exp}(\epsilon X_1), x + \epsilon X_2) = (g, x)$ in $G \times_B \mathfrak{n}$ 
provided that $X_1 \in \mathfrak{b}$ (ie. $\text{exp}(\epsilon X_1) \in B$), and 
$$
\text{exp}( \epsilon X_1) (x + \epsilon X_2) \text{exp}( - \epsilon X_1) \approx x
$$ 
Discarding non-linear powers of $\epsilon$, the latter translates to $x + \epsilon(X_2 + [X_1, x]) = x$, i.e. $X_2 = - [X_1, x]$. 
Thus the kernel of the map $\mathfrak{g} \oplus \mathfrak{n} = T_{(g, x)}(G \times \mathfrak{n}) \twoheadrightarrow T_{(g, x)}(G \times^B \mathfrak{n})$
is the subspace $\{ (X, -[X, x]) | X \in \mathfrak{b} \}$, so: 
$$
 T_{(g,x)} (G \times^B \mathfrak{n}) \simeq \frac{\mathfrak{g} \oplus \mathfrak{n}}{\{ (X, -[X, x]) \: | \: X \in \mathfrak{b} \}}. 
$$
Suppose $(g, x) \in G \times_B \mathfrak{n}$ lies in $U_n$; or equivalently, that 
$\tilde{x} := gxg^{-1} \in S_n$. Now given $(X_1, X_2) \in T_{(g,x)} (G \times_B \mathfrak{n})$, we have that 
$(X_1, X_2) \in T_{(g,x)} (U_n)$ when the curve $(g \cdot \text{exp}(\epsilon X_1), x + \epsilon X_2)$ lies in $U_n$. 
This happens precisely when $g \cdot \text{exp}(\epsilon X_1) (x + \epsilon X_2) \text{exp}( - \epsilon X_1) \cdot g^{-1} \in S_n$ (infinitesimally). 
Discarding non-linear powers of $\epsilon$, this is equivalent to saying that 
$$
g \cdot (x + \epsilon (X_2 + [X_1, x]) \cdot g^{-1} \in S_n.
$$
 Since $gxg^{-1} \in S_n$, this is equivalent to $X_2 + [X_1, x] \in g^{-1} \cdot C_n \cdot g$ (recall that $S_n = z_n + C_n$ where $C_n$ is a vector subspace). Thus: 
\begin{align*} 
T_{(g,x)} (U_n) \simeq & \frac{ \{ (X_1, X_2) \in \mathfrak{g} \oplus \mathfrak{n} \: | \: X_2 + [X_1, x] \in g^{-1} C_n g \} }{\{ (X, -[X, x]) \: | \: X \in \mathfrak{b} \}} \\
\simeq & \frac{ \{ (X, Y) \in \mathfrak{g} \oplus C_n \: | \: [X, \tilde{x}] + Y \in g \cdot \mathfrak{n} \} }{ g \cdot \mathfrak{b} \oplus 0 } 
\end{align*}
 For the last isomorphism, use the substitution $X = - g X_1 g^{-1}, Y = g (X_2 + [X_1, x]) g^{-1}$. Recall from the discussion in 
Section $1.4$ of \cite{mv} that the map $\pi: \mathfrak{g} \oplus C_n \rightarrow \mathfrak{g}, \pi(X, Y) = [X, \tilde{x}] + Y$ is surjective. Hence: 
\begin{align*} 
\text{dim}(T_{g,x}(U_n)) = \text{dim}(\mathfrak{n}) + \text{dim}(C_n) - \text{dim}(\mathfrak{b}) 
\end{align*}
In particular, this shows that $U_n$ is smooth. Now suppose that $(g, x) \in X_{n,i} \cap X_{n,j}$. 
It is clear that $X_{n,i} = U_n \cap (G \times^B \mathfrak{n}^i)$, where $\mathfrak{n}^i \subset \mathfrak{n}$ is the nilradical of the minimal parabolic 
corresponding to $i$. The above argument is valid after replacing $\mathfrak{n}$ with $\mathfrak{n}^i$, and we obtain: 
$$
T_{(g,x)} (X_{n,i}) \simeq \frac{ \{ (X, Y) \in \mathfrak{g} \oplus C_n \: | \: [X, \tilde{x}] + Y \in g \cdot \mathfrak{n}^i \} }{ g \cdot \mathfrak{b} \oplus 0 }
$$ 
$$
T_{(g,x)} (X_{n,j}) \simeq \frac{ \{ (X, Y) \in \mathfrak{g} \oplus C_n \: | \: [X, \tilde{x}] + Y \in g \cdot \mathfrak{n}^j \} }{ g \cdot \mathfrak{b} \oplus 0 }
$$ 
Using the surjectivity of $\pi$, it is clear that $T_{(g,x)} (X_{n,i})$ and $T_{(g,x)} (X_{n,j})$ are distinct co-dimension $1$ subspaces in 
$T_{(g,x)} (U_n)$. Hence $T_{(g,x)} (X_{n,i}) + T_{(g,x)} (X_{n,j}) = T_{(g,x)} (U_n)$, and $X_{n,i}$ and $X_{n,j}$ intersect transversely in $U_n$. 
\end{proof}

\begin{corollary} \label{transverse2} 
The following intersections are transverse: \begin{enumerate} \item $\pi_{12}^{-1}(X_{n, i}) \cap \pi_{23}^{-1}(X_{n,j})$ inside $U_{n-1} \times U_n \times U_{n-1}$ for $i \neq j$. \item $\pi_{12}^{-1}(X_{n,i}) \cap \pi_{23}^{-1}(X_{n+1,j})$ inside $U_{n-1} \times U_n \times U_{n+1}$.  
\end{enumerate} 
\end{corollary} 
\begin{proof} Both statements follow using Lemma $5.3$ from \cite{ck}; for the first, we also need Lemma \ref{lemma-transverse}. \end{proof}

\section{Tangles}\label{section-tangles}
\subsection{Affine tangles}


\begin{definition} If $p \equiv q \pmod 2$, a $(p,q)$ affine tangle is an embedding of $\frac{p+q}{2}$ arcs and a finite number of circles into the region $\{ (x,y) \in \mathbb{C} \times \mathbb{R} | 1 \leq |x| \leq 2 \}$, such that the end-points of the arcs are $(1,0),(\zeta_p, 0), \cdots , (\zeta_p^{p-1}, 0), (2, 0), (2\zeta_q,0), \cdots, (2 \zeta_q^{q-1},0)$ in some order; here $\zeta_k = e^{\frac{2\pi i}{k}}$. 
\end{definition}

\begin{remark} Given a $(p,q)$ affine tangle $\alpha$, and a $(q,r)$ affine tangle $\beta$,  
we can compose them using scaling and concatenation. This composition is associative up to isotopy.
The composition $\beta \circ \alpha$ is a $(p,r)$ affine tangle. \end{remark}

\begin{definition} Given $1 \leq i \leq n$, define the following affine tangles: \begin{itemize} \item Let $g_n^i$ denote the $(n-2, n)$ tangle with an arc connecting $(2\zeta_n^i,0)$ to $(2\zeta_n^{i+1},0)$. Let other strands connect $(\zeta_{n-2}^k,0)$ to $(2\zeta_n^k,0)$ for $1\leq k<i$ and $(\zeta_{n-2}^k,0)$ to $(2\zeta_n^{k+2},0)$ for $i+1<k\leq n-2$.
\item Let $f_n^i$ denote the $(n, n-2)$ tangle with an arc connecting $(\zeta_{n}^i,0)$ and $(\zeta_{n}^{i+1},0)$. Let other strands connect $(\zeta_{n}^k,0)$ to $(2\zeta_{n-2}^k,0)$ for $1\leq k<i$ and $(\zeta_{n}^k,0)$ to $(2\zeta_{n-2}^{k-2},0)$ for $i+1<k\leq n-2$.
\item Let $t_n^i(1)$ (respectively, $t^i_n(2)$) denote the $(n,n)$ tangle in which a strand connecting $(\zeta_n^{i},0)$ to $(2\zeta_{n}^{i+1},0)$ passes above (respectively, beneath) a strand connecting $(\zeta_n^{i+1},0)$ to $(2\zeta_n^{i},0)$. Let other strands connect $(\zeta_n^k,0)$ to $(2\zeta_n^k,0)$ for $k \ne i, i+1$.
 \item Let $r_n$ denote the $(n,n)$ tangle connecting $(\zeta_n^j, 0)$ to $(2\zeta_n^{j-1}, 0)$ for each $1 \leq j \leq n$ (clockwise rotation of all strands), and let $r_n'$ denote the $(n,n)$ tangle connecting $(\zeta_n^j, 0)$ to $(2\zeta_n^{j+1}, 0)$ for each $1 \leq j \leq n$ (counterclockwise rotation). 
\end{itemize} 
\end{definition}
The figure below has diagrams depicting some of these elementary tangles; see $s_4^4$ is defined below in Definition \ref{sni}. 

\includegraphics[scale=0.65]{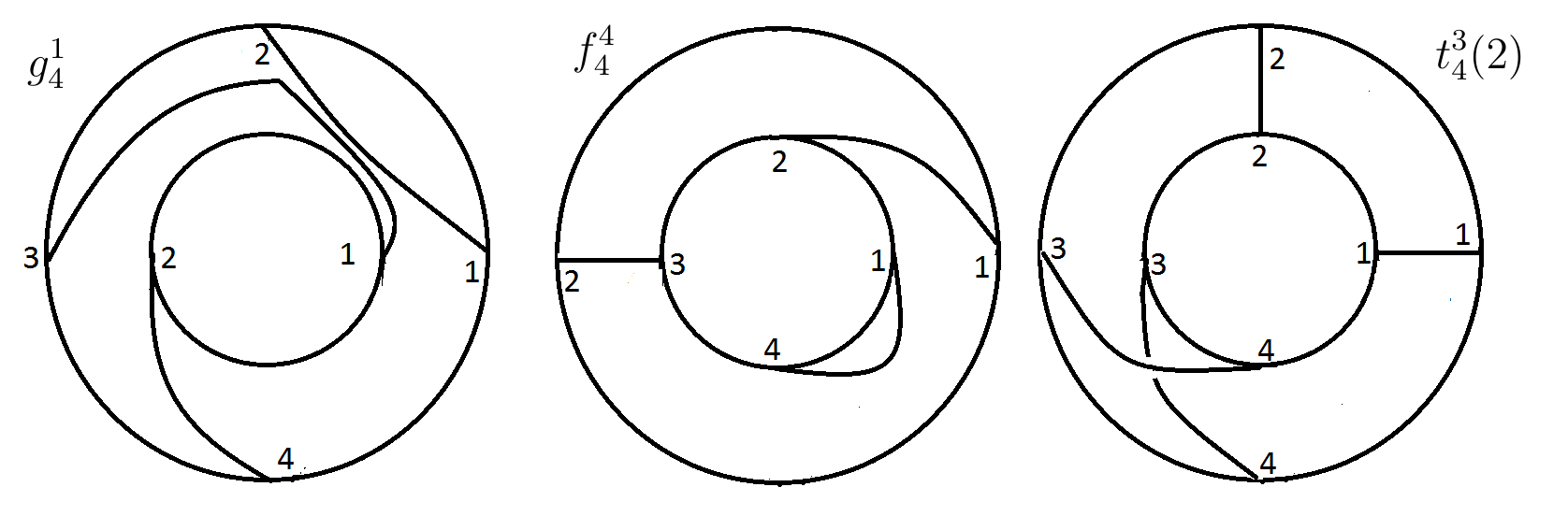}

\includegraphics[scale=0.65]{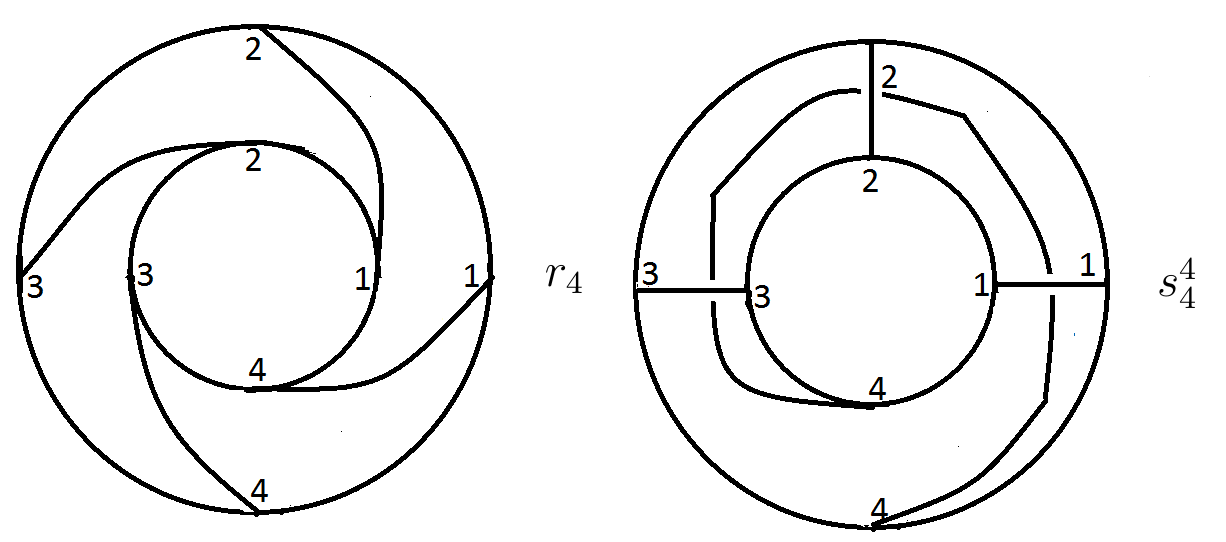}

\begin{definition} Define a linear tangle to be an affine tangle that is isotopic to a product of the generators $g_n^i, f_n^i, t_n^i(1)$ and $t_n^i(2)$ for $i \ne n$. \end{definition}

\begin{remark} Linear tangles can be moved away from the half-line $e^{i\epsilon}\mathbb{R}_{\geq 0}$ where $\epsilon$ is a small positive number. If we cut the annulus $1\leq |z|\leq 2$ by that line and apply the logarithm map, linear tangles turn into the usual tangles that live between two parallel lines.
\end{remark}

\begin{lemma}\label{lemma-generators}
Any affine tangle is isotopic to a composition of the above generators.
\end{lemma}
\begin{proof}
For a curve in $\mathbb{C}$, define its affine critical point as a point where this curve is tangent to a circle with center at $0$.
We can adjust a tangle within its isotopy class so that its projection onto $\mathbb{C}$ has a finite number of transversal crossings and affine critical points. We can also assume that no two of these points lie on the same circle with center at $0$. Cut the projection of the tangle by circles with center at $0$ into annuli so that each annulus contains only one crossing or affine critical point. We can further adjust the tangle so that we have a tangle inside each annulus, and by construction these tangles have to be $g_n^i$, $f_n^i$, or $t_n^i(p)$, possibly composed with a power of $r_n$. 
\end{proof}

\begin{definition}\label{DefinitionATan} Let $\textbf{ATan}$ (resp. $\textbf{Tan}$) denote the category with objects $k$ for $k \in \mathbb{Z}_{\geq 0}$, and the set of morphisms between $p$ and $q$ consist of all affine (resp. linear) $(p, q)$ tangles. 
\end{definition}

 In the category $\textbf{ATan}$ we record the following relations between the above generators; here let $1 \leq i \leq n-1, 1 \leq p,q \leq 2, k \geq 2$:
\begin{enumerate} 
\item\label{ATanMovesFirst} (Reidemeister $0$) $f_n^i \circ g_n^{i+1} = f_n^{i+1} \circ g_n^i = \text{id}$  
\item (Reidemeister $1$) $f_n^i \circ t_n^{i \pm 1}(2) \circ g_n^i = f_n^i \circ t^{i \pm 1}_n(1) \circ g_n^i = \text{id}$
\item (Reidemeister $2$) $t_n^i(1) \circ t^i_n(2) = t^i_n(2) \circ t_n^i(1) = id$ 
\item (Reidemeister $3$) $t_n^i(1) \circ t_n^{i+1}(1) \circ t_n^i(1) = t_n^{i+1}(1) \circ t_n^i(1) \circ t_n^{i+1}(1)$. 
\item (Cup-cup isotopy) $g_{n+2}^{i+k} \circ g_n^i = g_{n+2}^i \circ g_n^{i+k-2}$ 
\item (Cap-cap isotopy) $f_n^{i+k-2} \circ f_{n+2}^i = f_n^i \circ f_{n+2}^{i+k}$ 
\item (Cup-cap isotopy) $g_n^{i+k-2} \circ f_n^i = f_{n+2}^i \circ g_{n+2}^{i+k}, g_n^i \circ f_n^{i+k-2} = f_{n+2}^{i+k} \circ g_{n+2}^i$ 
\item (Cup-crossing isotopy) $g_n^i \circ t_{n-2}^{i+k-2}(q) = t_n^{i+k}(q) \circ g_n^i, g_n^{i+k} \circ t_{n-2}^i(q) = t_n^i(q) \circ g_n^{i+k}$ 
\item (Cap-crossing isotopy) $f_n^i \circ t_n^{i+k}(q) = t_{n-2}^{i+k-2}(q) \circ f_n^i, f_n^{i+k} \circ t_n^i(q) = t_{n-2}^{i}(q) \circ f_n^{i+k}$ 
\item (Crossing-crossing isotopy) $t_n^i(p) \circ t_n^{i+k}(q) = t_n^{i+k}(q) \circ t_n^{i}(p)$ 
\item\label{ATanMovesLastLinear} (Pitchfork move) $t_n^i(1) \circ g_{n}^{i+1} = t_n^{i+1}(2) \circ g_n^i, t_n^i(2) \circ g_n^{i+1} = t_n^{i+1}(1) \circ g_n^i$. 
\item (Rotation) $r_n \circ r_n' = r_n' \circ r_n = id$ 
\item \label{ATanMovesFirstAffine}(Cap rotation) $r_{n-2}' \circ f_n^i \circ r_n = f_n^{i+1}, f_n^{n-1} \circ r_n^2 = f_n^1$ 
\item (Cup rotation) $r_n' \circ g_n^i \circ r_{n-2} = g_n^{i+1}, r_n'^2 \circ g_n^{n-1} = g_n^1$ 
\item\label{ATanMovesLast} (Crossing rotation) $r_n' \circ t_n^i(q) \circ r_n = t_n^{i+1}(q), r_n'^2 \circ t_n^{n-1}(q) \circ r_n^2 = t_n^1(q)$. \end{enumerate}

By Lemma $4.1$ from \cite{ck}, any relation between linear tangles can be expressed as a composition of the relations
\eqref{ATanMovesFirst}-\eqref{ATanMovesLastLinear} above. We can generalize that to affine tangles:

\begin{proposition} \label{suff} Any relation between affine tangles can be expressed as a composition of the relations 
\eqref{ATanMovesFirst}-\eqref{ATanMovesLast} above.\end{proposition}
\begin{proof} First, let us reduce any relation to a composition of relations \eqref{ATanMovesFirst}-\eqref{ATanMovesLastLinear} involving $g_n^i$, $f_n^i$, $t_n^i(p)$ for $1\leq i\leq n$ (for the definition of $g_n^n$, $f_n^n$, $t_n^n(p)$ see the proof of Lemma \ref{lemma-generators}). Then, we can express the relations \eqref{ATanMovesFirst}-\eqref{ATanMovesLastLinear} involving $g_n^n$, $f_n^n$, $t_n^n(p)$ using relations \eqref{ATanMovesFirst}-\eqref{ATanMovesLast}, by a direct computation.

Let us call an isotopy {\it linear} if it fixes a segment of the form $[(\zeta, 0), (2\zeta,0)]$ for some $\zeta$. Note that a linear isotopy is a composition of elementary isotopies \eqref{ATanMovesFirst}-\eqref{ATanMovesLastLinear} (possibly involving $g_n^n$, $f_n^n$, $t_n^n(p)$) since the points where the tangle intersects  $[(\zeta, 0), (2\zeta,0)]$ stay fixed. Now, if two affine tangles are isotopic, then they are also isotopic through a composition of two linear isotopies, which completes the proof.
\end{proof}

For our purposes, it will be more convenient to replace the relations \eqref{ATanMovesFirstAffine}-\eqref{ATanMovesLast} by the equivalent set of defining relations below.

\begin{definition} \label{sni} Let $s_n^{i}$ denote the $(n,n)$-tangle with a strand connecting $(\zeta_j,0)$ to $(2 \zeta_j,0)$ for each $j$, and a strand connecting $(\zeta_i,0)$ to $(2 \zeta_i,0)$ passing clockwise around the circle, beneath all the other strands. \end{definition}
\begin{lemma} \label{affrelns} 
The following relations are equivalent to the relations \eqref{ATanMovesFirstAffine}-\eqref{ATanMovesLast} above. 
\begin{itemize} 
\item $s_n^n \circ g_{n}^i = g_n^i \circ s_{n-2}^{n-2}$, $\ s_{n-2}^{n-2} \circ f_n^i = f_n^i \circ s_n^n$, $\ s_n^n \circ t_n^i(p) = t_n^i(p) \circ s_n^n$; 
\item $f_{n}^{n-1} \circ s_n^n \circ t_{n}^{n-1}(2) \circ s_n^n \circ t_n^{n-1}(2)=f_n^{n-1}$; 
\item $s_n^n \circ t_n^{n-1}(2) \circ s_n^n \circ t_n^{n-1}(2) \circ g_n^{n-1} = g_n^{n-1}$; 
\item $t_n^{n-1}(2) \circ s_n^n \circ t_n^{n-1}(2) \circ s_n^n \circ t_n^{n-1}(2) = s_n^n \circ t_n^{n-1}(2) \circ s_n^n \circ t_n^{n-1}(2) \circ t_n^{n-1}(2)$.
\end{itemize} 
\end{lemma} 
\begin{proof} 
It is straightforward to verify the  that we have the relation $r_n=s_n^n \circ t_n^{n-1}(2) \circ \cdots \circ t_n^1(2)$. It remains to see that the relations 
\eqref{ATanMovesFirstAffine}-\eqref{ATanMovesLast} then follow from the those listed in the statement of this Lemma, and the
relations \eqref{ATanMovesFirst}-\eqref{ATanMovesLastLinear}. This can be done by direct computation; as an example, see the below calculation for relation (13). 
\begin{align*} r'_{n-2} \circ f_n^i \circ r_n &= t_{n-2}^1(1) \circ \cdots \circ t_{n-2}^{n-3}(1) \circ (s_{n-2}^{n-2})^{-1} \circ f_n^i \circ s_n^n \circ t_n^{n-1}(2) \circ \cdots \circ t_n^1(2) \\ &= t_{n-2}^1(1) \circ \cdots \circ t_{n-2}^{n-3}(1) \circ (s_{n-2}^{n-2})^{-1} \circ s_{n-2}^{n-2} \circ f_n^i  \circ t_n^{n-1}(2) \circ \cdots \circ t_n^1(2) \\ &= t_{n-2}^1(1) \circ \cdots \circ t_{n-2}^{n-3}(1) \circ f_n^i \circ t_n^{n-1}(2) \circ \cdots \circ t_n^1(2) = f_n^{i+1} \end{align*} \end{proof}

\subsection{Framed tangles}


All preceding constructions may be carried out for framed tangles. Define the generators
$\hat{g}_n^i$ (resp. $\hat{f}_n^i$, resp. $\hat{t}_n^i(l)$, resp. $\hat{r}_n^i$) as
tangles $g_n^i$ (resp. ${f}_n^i$, resp. ${t}_n^i(l)$, resp. ${r}_n^i$) with blackboard framing. Introduce new generators $\hat{w}_n^i(1)$ and $\hat{w}_n^i(2)$, which correspond to positive and negative twists of framing of the $i$th strand of an $(n,n)$ identity tangle. 
\begin{definition} Define a framed linear tangle to be a framed affine tangle that isotopic to a product of the generators $\hat{g}_n^i, \hat{f}_n^i, \hat{t}_n^i(1), \hat{t}_n^i(2)$ for $i \neq n$, and $\hat{w}_n^i(1), \hat{w}_n^i(2)$. \end{definition}
\begin{definition} \label{AFTanRelations}Consider the category $\textbf{AFTan}$ (resp. $\textbf{FTan}$), with objects $k$ for $k \in \mathbb{Z}_{\geq 0}$, and the set of morphisms between $p$ and $q$ consist of all framed affine (resp. framed linear) $(p,q)$ tangles. \end{definition}
The relations for framed tangles are transformed as follows: 
\begin{enumerate} 
\item\label{AFTanMovesFirst} $\hat{f}_n^i\circ \hat{g}_n^{i+1} = id = \hat{f}_n^{i+1} \circ \hat{g}_n^i$ 
\item\label{AFTanReidemeister1} (Reidemeister 1) $\hat{f}_n^i\circ \hat{t}_n^{i\pm 1}(l)\circ \hat{g}_n^i = \hat{w}_n^i(l)$ 
\item\label{AFTanMovesThird} $\hat{t}_n^i(2)\circ \hat{t}_n^i(1) = id = \hat{t}_n^i(1)\circ \hat{t}_n^i(2)$ 
\item\label{AFTanReidemeister3} $\hat{t}_n^i(l)\circ \hat{t}_n^{i+1}(l)\circ \hat{t}_n^i(l) =  \hat{t}_n^{i+1}(l)\circ \hat{t}_n^{i}(l)\circ \hat{t}_n^{i+1}(l)$ 
\item $\hat{g}_{n+2}^{i+k}\circ \hat{g}_n^i = \hat{g}_{n+2}^i\circ \hat{g}_n^{i+k-2}$ 
\item $\hat{f}_{n}^{i+k-2}\circ \hat{f}_{n+2}^i = \hat{f}_{n}^i\circ \hat{f}_{n+2}^{i+k}$ 
\item $\hat{g}_n^{i+k-2}\circ \hat{f}_n^i = \hat{f}_{n+2}^i\circ \hat{g}_{n+2}^{i+k}, \quad \hat{g}_n^{i}\circ \hat{f}_n^{i+k-2} = \hat{f}_{n+2}^{i+k}\circ \hat{g}_{n+2}^{i}$
\item\label{AFTanPitchfork} $\hat{g}_n^i\circ \hat{t}_{n-2}^{i+k-2}(l) = \hat{t}_n^{i+k}(l)\circ \hat{g}_n^i, \quad \hat{g}_n^{i+k}\circ \hat{t}_{n-2}^{i}(l) = \hat{t}_n^{i}(l)\circ \hat{g}_n^{i+k}$ 
\item $\hat{f}_n^i\circ \hat{t}_{n}^{i+k}(l) = \hat{t}_{n-2}^{i+k-2}(l)\circ \hat{f}_n^i, \quad \hat{f}_n^{i+k}\circ \hat{t}_{n}^{i}(l) = \hat{t}_{n-2}^{i}(l)\circ \hat{f}_n^{i+k}$ 
\item $\hat{t}_n^i(l)\circ \hat{t}_n^{i+k}(m) =  \hat{t}_n^{i+k}(m)\circ \hat{t}_n^i(l)$ 
\item\label{TanMovesLast}  $\hat{t}_n^i(1)\circ \hat{g}_n^{i+1} = \hat{t}_n^{i+1}(2)\circ \hat{g}_n^i,\quad \hat{t}_n^i(2)\circ \hat{g}_n^{i+1}= \hat{t}_n^{i+1}(1)\circ \hat{g}_n^i$ 
\item $\hat{r}_n\circ \hat{r}_n'  = id = \hat{r}_n' \circ \hat{r}_n$ 
\item $\hat{r}_{n-2}'\circ \hat{f}_n^i \circ \hat{r}_n  = \hat{f}_n^{i+1},\ i=1,\ldots, n-2; \quad
 \hat{f}_n^{n-1} \circ (\hat{r}_n)^2  = \hat{f}_n^1$ 
\item $\hat{r}_n' \circ \hat{g}_n^i \circ \hat{r}_{n-2}  = \hat{g}_n^{i+1}, \ i=1,\ldots, n-2; \quad
 (\hat{r}_n')^{2} \circ \hat{g}_n^{n-1} = \hat{g}_n^1$ 
\item  $\hat{r}_n' \circ \hat{t}_n^i(l)\circ \hat{r}_n  = \hat{t}_n^{i+1}(l); \quad
 (\hat{r}_n')^{2}\circ \hat{t}_n^{n-1}(l) \circ (\hat{r}_n)^2  = \hat{t}_n^1(l)$ 

We have the following additional relations for twists:

\item\label{AFTanMovesFirstTwist} $\hat{w}_n^i(1) \circ \hat{w}_n^i(2) = id, \quad \hat{w}_n^i(l) \circ \hat{w}_n^j(k) = \hat{w}_n^j(k) \circ \hat{w}_n^i(l),\ i\ne j$ 
\item $\hat{w}_n^i(k) \circ \hat{g}_n^i = \hat{w}_n^{i+1}(k) \circ \hat{g}_n^i, \quad%
 \hat{w}_n^i(k) \circ \hat{g}_n^j =  \hat{g}_n^j \circ \hat{w}_n^{i+1\pm 1}(k),\ i\ne j, j+1$
\item $\hat{f}_n^i \circ \hat{w}_n^i(k) = \hat{f}_n^i \circ \hat{w}_n^{i+1}(k), \quad%
\hat{w}_n^i(k) \circ \hat{f}_n^j =  \hat{f}_n^j \circ \hat{w}_n^{i-1\pm 1}(k),\ i\ne j, j+1$
\item $\hat{w}_n^i(k) \circ \hat{t}_n^i = \hat{w}_n^{i+1}(k) \circ \hat{t}_n^i, \quad%
 \hat{w}_n^i(k) \circ \hat{t}_n^j =  \hat{t}_n^j \circ \hat{w}_n^{i}(k),\ i\ne j, j+1$ 
\item \label{FTanMovesLast} $\hat{t}_n^i \circ \hat{w}_n^i(k) = \hat{f}_n^i \circ \hat{w}_n^{i+1}(k), \quad%
\hat{w}_n^i(k) \circ \hat{f}_n^j =  \hat{t}_n^j \circ \hat{w}_n^{i}(k),\ i\ne j, j+1$
\item\label{AFTanMovesLast} $\hat{w}_n^i(k) \circ \hat{r}_n = \hat{r}_n \circ \hat{w}_n^{i-1}(k), \quad%
 \hat{w}_n^i(k) \circ \hat{r}'_n = \hat{r}'_n \circ \hat{w}_n^{i+1}(k)$ \end{enumerate} 

Note how the Reidemeister 1 move \eqref{AFTanReidemeister1} is the only relation between the non-twist generators that differs from the relations in $\textbf{ATan}$.

\begin{proposition} 
Any isotopy of affine framed tangles is equivalent to a composition of elementary isotopies \eqref{AFTanMovesFirst}-\eqref{AFTanMovesLast}.
\end{proposition}
\begin{proof} There is a forgetful functor from the $2$-category of framed tangles and their isotopies to the 
$2$-category of non-framed tangles and their isotopies, which forgets the framing. 
Thus, for every isotopy there is a composition of relations \eqref{ATanMovesFirst}-\eqref{ATanMovesLast} (in $\textbf{ATan}$) which differs only in framing, and that can be ruled out by the commutation laws \eqref{AFTanMovesFirstTwist}-\eqref{AFTanMovesLast} (in $\textbf{AFTan}$) of twists with all other generators. \end{proof}

Lemma \ref{affrelns} still holds in this context, after replacing $s_n^n$ by it's ``framed" version $\hat{s}_n^n$.

\section{Functors associated to affine tangles}\label{section-functors}


\begin{definition} Recall that $\textbf{AFTan}$ (resp $\textbf{Tan}$, $\textbf{FTan}$) has objects $\{ k \}$ for $k \in \mathbb{Z}_{\geq 0}$, and the set of morphisms between $\{ p \}$ and $\{ q \}$ consists of all framed affine (resp. framed linear) $(p, q)$ tangles. Define the category $\textbf{AFTan}_m$ (resp. $\textbf{Tan}_m, \textbf{FTan}_m$) to be the full subcategory of $\textbf{AFTan}$ (resp. $\textbf{Tan}, \textbf{FTan}$) with objects $\{ m+2k \}$ for $k \in \mathbb{Z}_{\geq 0}$. \end{definition} \begin{definition} A ``weak representation" of the category $\textbf{AFTan}_m$ is an assignment of a triangulated category $\mathcal{C}_k$ for each $k \in \mathbb{Z}_{\geq 0}$, and a functor $\Psi(\alpha): \mathcal{C}_p \rightarrow \mathcal{C}_q$ for each framed affine $(m+2p, m+2q)$-tangle, so that the relations between tangles hold for these functors: i.e. if $\beta$ is an $(m+2q, m+2r)$ tangle, then there is an isomorphism $\Psi(\beta) \circ \Psi(\alpha) \simeq \Psi(\beta \circ \alpha)$. \end{definition}

Similarly one can define the notion of a ``weak representation" of the categories $\textbf{Tan}_m, \textbf{FTan}_m$.
The goal of this section is to construct a weak representation of $\textbf{AFTan}_m$ using the categories $\mathcal{D}_k$. 

In \cite{ck} Cautis and Kamnitzer construct a weak representation of the category of oriented tangles. We are going to adapt their construction to our setting of framed tangles, and then generalize it to the category $\textbf{AFTan}_m$ of affine framed tangles. The relations between the generators for oriented tangles are mostly the same as the relations we use here, with a notable exception of Reidemeister I move.

\subsection{Cautis and Kamnitzer's representation of the oriented tangle calculus} Let $\widetilde{\mathcal{D}}_n=D^b(\text{Coh}(Y_{m+2n}))$. In section $4$ of \cite{ck}, Cautis and Kamnitzer construct a weak representation of the category $\textbf{OTan}_m$ of oriented tangles using the categories $\widetilde{\mathcal{D}}_n$.  In fact, Cautis and Kamnitzer construct a weak representation of the full category $\textbf{OTan}$ (which gives a weak representation of the subcategory $\textbf{OTan}_m$).
Also, Cautis and Kamnitzer deal with the $\mathbb{C}^*$-equivariant derived categories; but we will omit this $\mathbb{C}^*$-equivariance as we do not need it. In this subsection we are going to recall their construction, altered so that it becomes a weak representation of $\textbf{FTan}$.

Recall the definition of Fourier-Mukai transforms (see \cite{huybr} for an extended treatment). Here all pullbacks, pushforwards, Homs and tensor products of sheaves will denote the corresponding derived functors. 

\begin{definition} (\cite{huybr}) Let $X, Y$ be two complex algebraic varieties, and let $\pi_1: X \times Y \rightarrow X, \pi_2: X \times Y \rightarrow Y$ denote the two projections. For an object  $\mathcal{T} \in D^b(\text{Coh}(X \times Y))$, define the Fourier-Mukai transform $\Psi_{\mathcal{T}}: D^b(\text{Coh}(X)) \rightarrow D^b(\text{Coh}(Y))$ by $\Psi_{\mathcal{T}}(\mathcal{F}) =  \pi_{2*}(\pi_{1}^* \mathcal{F} \otimes \mathcal{T})$. The object $\mathcal{T}$ is then called the Fourier-Mukai kernel of $\Psi_{\mathcal{T}}$. \end{definition}

 Let $\widetilde{\mathcal{V}}_k$ denote the tautological vector bundle on $Y_{m+2n}$ corresponding to $V_k$, and let $\widetilde{\mathcal{E}}_k$ be  the quotient line bundle $\widetilde{\mathcal{E}}_k = \widetilde{\mathcal{V}}_k/\widetilde{\mathcal{V}}_{k-1}$. 
The following two definitions are based on \cite{ck}, but not identical to the definitions there:

\begin{definition}  Define the following Fourier-Mukai kernels: \begin{align*} \widetilde{\mathcal{G}}_{m+2n}^i &= \mathcal{O}_{X_{m+2n}^i} \otimes \pi_2^* \widetilde{\mathcal{E}}_i \in D^b(\text{Coh}(Y_{m+2n-2} \times Y_{m+2n})), \\ \widetilde{\mathcal{F}}_{m+2n}^i &= \mathcal{O}_{X_{m+2n}^i} \otimes \pi_1^* \widetilde{\mathcal{E}}_{i+1}^{-1} \in D^b(\text{Coh}(Y_{m+2n} \times Y_{m+2n-2})) \\ \widetilde{\mathcal{T}}_{m+2n}^i(1) &= \mathcal{O}_{Z_{m+2n}^i} \in D^b(\text{Coh}(Y_{m+2n} \times Y_{m+2n})) \\ \widetilde{\mathcal{T}}_{m+2n}^i(2) &= \mathcal{O}_{Z_{m+2n}^i} \otimes \pi_1^* {\widetilde{\mathcal{E}}_{i+1}^{-1}} \otimes \pi_2^*{\widetilde{\mathcal{E}}_i} \in D^b(\text{Coh}(Y_{m+2n} \times Y_{m+2n})) \end{align*} 
\end{definition}

\begin{definition} Define the functors 
\begin{align*}\widetilde{G}_{m+2n}^i & = \widetilde{\Psi}(g_{m+2n}^{i}) = \Psi_{\widetilde{\mathcal{G}}_{m+2n}^i}  : \widetilde{\mathcal{D}}_{n-1} \rightarrow \widetilde{\mathcal{D}}_n\\
\widetilde{F}_{m+2n}^i & = \widetilde{\Psi}(f_{m+2n}^{i}) = \Psi_{\widetilde{\mathcal{F}}_{m+2n}^i}  : \widetilde{\mathcal{D}}_{n} \rightarrow \widetilde{\mathcal{D}}_{n-1} \\ \widetilde{T}_{m+2n}^i(1) & =  \widetilde{\Psi}(t_{m+2n}^{i}(1)) = \Psi_{\widetilde{\mathcal{T}}_{m+2n}^i(1)}  : \widetilde{\mathcal{D}}_{n} \rightarrow \widetilde{\mathcal{D}}_{n}\\ \widetilde{T}_{m+2n}^i(2) & = \widetilde{\Psi}(t_{m+2n}^{i}(2)) = \Psi_{\widetilde{\mathcal{T}}_{m+2n}^i(2)}  : \widetilde{\mathcal{D}}_{n} \rightarrow \widetilde{\mathcal{D}}_{n} \end{align*} \begin{align*}
\widetilde{W}_{m+2n}^i(1) & = \widetilde{\Psi}(w_{m+2n}^{i}(1)) = [-1]  : \widetilde{\mathcal{D}}_{n} \rightarrow \widetilde{\mathcal{D}}_{n} \\
\widetilde{W}_{m+2n}^i(2) & = \widetilde{\Psi}(w_{m+2n}^{i}(2)) = [1]  : \widetilde{\mathcal{D}}_{n} \rightarrow \widetilde{\mathcal{D}}_{n} 
\end{align*} \end{definition}

Note that the difference with the definition in \cite{ck} is that we only use two kinds of twists $\widetilde{T}(1)$ and $\widetilde{T}(2)$ where they use four, and our twists differ from their twists by a shift. The reasons for this change are, first, that there are only two different crossing generators in the category $\textbf{FTan}$ while there are four in 
$\textbf{OTan}$; second, this is the change that turns the oriented tangle relations into the framed tangle relations (see Proposition \ref{tangle} below); and third, it gives us the skein relation in a nice form of an exact triangle
$Id \to \widetilde{\Psi}(t_n^i(2))\to \widetilde{\Psi}(g_n^i\circ f_n^i)$ in the spirit of Khovanov's homology construction as described in \cite{khovanov} (see Lemma \ref{adjoint} below).

The functors $\widetilde{G}_{m+2n}^i: \widetilde{\mathcal{D}}_{n-1} \rightarrow \widetilde{\mathcal{D}}_{n}$ admit the following alternate description: $\widetilde{G}_{m+2n}^i(\mathcal{F})=j_*(p^* \mathcal{F} \otimes \widetilde{\mathcal{E}}_i)$ for $\mathcal{F} \in \widetilde{\mathcal{D}}_{n-1}$. Similarly, the functor $\widetilde{F}_{m+2n}^i: \widetilde{\mathcal{D}}_{n} \rightarrow \widetilde{\mathcal{D}}_{n-1}$ admits the following description: $\widetilde{F}_{m+2n}^i(\mathcal{G})=p_*(j^* \mathcal{G} \otimes \widetilde{\mathcal{E}_{i+1}^{-1}})$ for $\mathcal{G} \in \widetilde{\mathcal{D}}_n$. The following calculation of the left and right adjoints to $\widetilde{G}_{m+2n}^i$, and an alternative description of the functors $\widetilde{T}_{m+2n}^i(1), \widetilde{T}_{m+2n}^i(2)$, from \cite{ck} will be of use to us. 

\begin{lemma} \label{adjoint} We have $(\widetilde{G}_{m+2n}^i)^R = \widetilde{F}_{m+2n}^i[-1]$ and $(\widetilde{G}_{m+2n}^i)^L = \widetilde{F}_{m+2n}^i[1]$. Also, for $\mathcal{F} \in \mathcal{D}_n$, there are distinguished triangles $\widetilde{G}_{m+2n}^i (\widetilde{G}_{m+2n}^i)^R \mathcal{F} \rightarrow \mathcal{F} \rightarrow \widetilde{T}_{m+2n}^i(2) \mathcal{F}$ and $\widetilde{T}_{m+2n}^i(1) \mathcal{F} \rightarrow \mathcal{F} \rightarrow \widetilde{G}_{m+2n}^i (\widetilde{G}_{m+2n}^i)^L$. \end{lemma} \begin{proof} This follows from Lemma $4.4$, and Theorem $4.6$ in \cite{ck}. \end{proof}

Recall that any framed linear tangle can be expressed as a composition of the above generators, and that any relation between linear tangles can be expressed via the relations \eqref{AFTanMovesFirst}-\eqref{TanMovesLast}, \eqref{AFTanMovesFirstTwist}-\eqref{FTanMovesLast} in Definition \ref{AFTanRelations}. Hence defining functors $\Psi(\alpha)$ for each $(m+2p, m+2q)$-tangle $\alpha$, which are compatible under composition, is equivalent to defining functors for each of the generators, satisfying the relations \eqref{AFTanMovesFirst}-\eqref{TanMovesLast}, \eqref{AFTanMovesFirstTwist}-\eqref{FTanMovesLast} (up to isomorphism). 

\begin{proposition} \label{tangle} The functors $\widetilde{\Psi}(f_{m+2n}^i), \widetilde{\Psi}(g_{m+2n}^i), \widetilde{\Psi}(t_{m+2n}^i(l)),  \widetilde{\Psi}(w_{m+2n}^{i}(l))$ satisfy the relations \eqref{AFTanMovesFirst}-\eqref{TanMovesLast}, \eqref{AFTanMovesFirstTwist}-\eqref{FTanMovesLast}. Thus, given a linear $(m+2p, m+2q)$ tangle, $\alpha$, written as a product of generators, we can define $\widetilde{\Psi}(\alpha)$ by composition (and up to isomorphism, the result does not depend on the choice of decomposition as a product of generators). This gives a weak representation of $\textbf{FTan}_m$ using the categories $\widetilde{\mathcal{D}}_{n}$. \end{proposition} 

\begin{proof} 
By Theorem $4.2$ in \cite{ck}, the functors $\widetilde{G}_{m+2n}^i$, $\widetilde{F}_{m+2n}^i$, $\widetilde{T}_{m+2n}^i(1)[1]$, and $\widetilde{T}_{m+2n}^i(2)[-1]$ satisfy the relations in the category $\textbf{OTan}$ that differ slightly from the relations \eqref{AFTanMovesFirst}-\eqref{TanMovesLast}. The relations \eqref{AFTanMovesFirst}, \eqref{AFTanMovesThird}-\eqref{TanMovesLast} are identical for $\textbf{OTan}$ and $\textbf{FTan}$, and they hold for the functors $\widetilde{G}_{m+2n}^i$, $\widetilde{F}_{m+2n}^i$, $\widetilde{T}_{m+2n}^i(1)$, $\widetilde{T}_{m+2n}^i(2)$ as well since every relation has the same number of each type of crossings on both sides, so after shifting every type $1$ crossing by $[1]$ and every type $2$ crossing by $[-1]$ the relations still hold. The oriented Reidemeister move I relation
$$\widetilde{F}_{m+2n}^i\circ \widetilde{T}_{m+2n}^{i\pm1}(1)[1]\circ \widetilde{G}_{m+2n}^i\simeq \mathrm{Id}\simeq \widetilde{F}_{m+2n}^i\circ \widetilde{T}_{m+2n}^{i\pm1}(2)[-1]\circ \widetilde{G}_{m+2n}^i$$
 is exactly the relation \eqref{AFTanReidemeister1} for $\widetilde{G}_{m+2n}^i$, $\widetilde{F}_{m+2n}^i$, $\widetilde{T}_{m+2n}^i(1)$, $\widetilde{T}_{m+2n}^i(2)$, and $\widetilde{W}_{m+2n}^i(l)$:
\begin{align*}
\widetilde{F}_{m+2n}^i\circ \widetilde{T}_{m+2n}^{i\pm1}(1)\circ \widetilde{G}_{m+2n}^i & \simeq  [-1]=\widetilde{W}_{m+2n}^i(1)\\
\widetilde{F}_{m+2n}^i\circ \widetilde{T}_{m+2n}^{i\pm1}(2)\circ \widetilde{G}_{m+2n}^i & \simeq  [1] = \widetilde{W}_{m+2n}^i(2)
\end{align*}
The relations  \eqref{AFTanMovesFirstTwist}-\eqref{FTanMovesLast} are straightforward.
\end{proof}


\subsection{Constructing functors $\Psi(\alpha): \mathcal{D}_p \rightarrow \mathcal{D}_q$ indexed by linear tangles: cups and caps}

In the previous section we constructed a weak representation of the category $\textbf{FTan}$ of framed tangles using the triangulated categories $\widetilde{\mathcal{D}}_n=D^b(\text{Coh}(Y_{m+2n}))$. Our next goal is to construct a weak representation of the category $\textbf{AFTan}$ of affine framed tangles using the categories $\mathcal{D}_n=D^b(\text{Coh}_{\mathcal{B}_{z_n}}(U_n))$. The embedding $i_n: U_n \rightarrow Y_{m+2n}$ induces a functor $i_{n *}: \mathcal{D}_n \rightarrow \widetilde{\mathcal{D}}_{n}$ for each $n$, thus one may hope to ``lift" the functor $\widetilde{\Psi}(\alpha): \widetilde{\mathcal{D}}_{p} \rightarrow \widetilde{\mathcal{D}}_{q}$ to a functor $\Psi(\alpha): \mathcal{D}_p \rightarrow \mathcal{D}_q$. In more precise terms, we aim to construct a functor $\Psi(\alpha)$ such that $i_{q *} \circ \Psi(\alpha) = \widetilde{\Psi}(\alpha) \circ i_{p *}$. 
Note that this isomorphism together with the isomorphism $\widetilde{\Psi}(\beta \circ \alpha) \simeq \widetilde{\Psi}(\beta) \circ \widetilde{\Psi}(\alpha)$ does not yet imply the isomorphism 
 $\Psi(\beta \circ \alpha) \simeq \Psi(\beta) \circ \Psi(\alpha)$, so we will need to prove the latter separately along with our construction of $\Psi(\alpha)$, employing an argument similar to one in \cite{ck}.

 Let $\mathcal{V}_k$ denote the tautological vector bundle on $S_n \times_{\mathfrak{sl}_{m+2n}} T^*\mathcal{B}_{n}$ corresponding to $V_k$, and let $\mathcal{E}_k$ be  the quotient line bundle $\mathcal{E}_k = \mathcal{V}_k/\mathcal{V}_{k-1}$. 

\begin{definition}\label{definition-kernels-G-F}
 Define the following Fourier-Mukai kernels:
\begin{align*} \mathcal{G}_{m+2n}^i &= \mathcal{O}_{X_{n,i}} \otimes \pi_2^* \mathcal{E}_i \in D^b(\text{Coh}(U_{n-1} \times U_n)), \\ \mathcal{F}_{m+2n}^i &= \mathcal{O}_{X_{n,i}} \otimes \pi_1^* \mathcal{E}_{i+1}^{-1} \in D^b(\text{Coh}(U_{n} \times U_{n-1})) \end{align*} 
\end{definition}

\begin{definition} Define the functors:
\begin{align*} G_{m+2n}^i = \Psi(g_{m+2n}^i) = \Psi_{\mathcal{G}_{m+2n}^i}: \mathcal{D}_{n-1} \rightarrow \mathcal{D}_n \\ F_{m+2n}^i = \Psi(f_{m+2n}^i) = \Psi_{\mathcal{F}_{m+2n}^i}: \mathcal{D}_{n} \rightarrow \mathcal{D}_{n-1} \end{align*} \end{definition}

\begin{remark} A priori, the functor $G_{m+2n}^i$ maps $D^b(\text{Coh}(U_{n-1}))$ to $D^b(\text{Coh}(U_n))$. However, it is easy to see that $G_{m+2n}^i$ maps the subcategory $\mathcal{D}_{n-1} = D^b(\text{Coh}_{\mathcal{B}_{z_{n-1}}}(U_{n-1})) \subset D^b(\text{Coh}(U_{n-1}))$  to the subcategory $\mathcal{D}_{n} = D^b(\text{Coh}_{\mathcal{B}_{z_n}}(U_n)) \subset D^b(\text{Coh}(U_{n-1}))$; similarly $F_{m+2n}^i$ maps $\mathcal{D}_n$ to $\mathcal{D}_{n-1}$. \end{remark}

The functors $G_{m+2n}^i : \mathcal{D}_{n-1} \rightarrow \mathcal{D}_n$ admit the following alternate description: $G_{m+2n}^i(\mathcal{F}) = j_{n,i*}( \pi_{n,i}^*\mathcal{F} \otimes \mathcal{E}_k)$ for $\mathcal{F} \in \mathcal{D}_{n-1}$. Similarly, the functor $F_{m+2n}^i : \mathcal{D}_{n} \rightarrow \mathcal{D}_{n-1}$ can be expressed as follows: $F_{m+2n}^i(\mathcal{G}) = \pi_{n,i*}(j_{n,i}^*\mathcal{G} \otimes \mathcal{E}_{k+1}^{-1})$ for $\mathcal{G} \in \mathcal{D}_n$. We will define the functors $\Psi(t_{m+2n}^i (1))$ and $\Psi(t_{m+2n}^i(2))$ in the next section, by proving an analogue of Lemma \ref{adjoint} above.


\subsection{Constructing functors $\Psi(\alpha): \mathcal{D}_p \rightarrow \mathcal{D}_q$ indexed by linear tangles: crossings and the framing}

Recall the definitions of spherical twists and spherical functors from \cite{ALDGSpherical}:
\begin{definition} Suppose we have two triangulated categories $\mathcal{C}$ and $\mathcal{D}$, and a functor $S: \mathcal{C} \rightarrow \mathcal{D}$, with a left adjoint $L: \mathcal{D} \rightarrow \mathcal{C}$ and a right adjoint $R: \mathcal{D} \rightarrow \mathcal{C}$. Assume that the categories $\mathcal{C}$ and $\mathcal{D}$ admit DG-enhancements, and the functors $S$, $R$, and $L$ descend from DG-functors between those (this holds for Fourier-Mukai transforms between derived categories of coherent sheaves, see \cite{ALDGSpherical} Example 4.3). Then the four adjunction maps for $(L, S, R)$ have canonical cones, and we can define these cones to be the twist $T_S(1)$, the dual twist $T_S(2)$, the cotwist $F_S(1)$, and the dual co-twist $F_S(2)$:
\begin{align*}
SR\to\mathrm{id}\to T_S(1); \qquad & T_S(2)\to\mathrm{id}\to SL;\\
F_S(1)\to\mathrm{id}\to RS; \qquad & LS \to \mathrm{id}\to F_S(2).
\end{align*}
\end{definition}

\begin{definition}\label{def-spherical-functors}
The functor $S$ is called spherical if the following four conditions hold:
\begin{enumerate}
\item\label{sphdefcond1} $T_S(1)$ and $T_S(2)$ are quasi-inverse autoequivalences of $\mathcal{D}$;
\item\label{sphdefcond2} $F_S(1)$ and $F_S(2)$ are quasi-inverse autoequivalences of $\mathcal{C}$;
\item\label{sphdefcond3} The composition $LT_S(1)[-1]\to LSR\to R$ of canonical maps is an isomorphism of functors;
\item\label{sphdefcond4} The composition $R\to RSL\to F_S(1)L[1]$ of canonical maps is an isomorphism of functors.
\end{enumerate}
\end{definition}

\begin{theorem}(\cite{ALDGSpherical}) Any two conditions in Definition \ref{def-spherical-functors} imply all four.
\end{theorem}

The usual way to prove that a functor is spherical is to use condition (\ref{sphdefcond2}) and one of the conditions (\ref{sphdefcond3}) and (\ref{sphdefcond4}). We are going to focus on functors for which a stronger version of (\ref{sphdefcond2}) holds:

\begin{definition} A spherical functor $S: \mathcal{C} \rightarrow \mathcal{D}$ is called strongly spherical if  $F_{S}(1)=[-3]$. \end{definition}

It turns out that if we use strongly spherical functors and their adjoints and twists to construct weak representations of $\textbf{FTan}_{m}$, the only relations we need to check are the Reidemeister 0 move and the commutation relations between non-adjacent cups and caps; all relations involving crossings follow automatically.

\begin{theorem} \label{spherical} Suppose we have a triangulated category $\mathcal{C}_{m+2k}$ for each $k \in \mathbb{Z}_{\geq 0}$; and for each $k \geq 1$, $1 \leq i < m+2k$, a strongly spherical functor $S_{m+2k}^i: \mathcal{C}_{m+2k-2} \rightarrow \mathcal{C}_{m+2k}$. Let $L_{m+2k}^i$ be it's left adjoint; $R_{m+2k}^i$ be it's right adjoint; $T_{m+2k}^i(1)$ its twist, and $T_{m+2k}^i (2)$ its dual twist. If the following conditions hold: 
\begin{enumerate} 
\item $S_{m+2k}^i L_{m+2k}^{i \pm 1} [-1] \simeq \mathrm{id}$ 
\item $S_{m+2k+2}^{i+l} S_{m+2k}^i \simeq S_{m+2k+2}^i S_{m+2k}^{i+l-2}$ for $l \geq 2$
\item $S_{m+2k}^{i+l-2} \circ L_{m+2k}^{i} \simeq L_{m+2k+2}^{i} \circ S_{m+2k+2}^{i+l}$, $S_{m+2k}^{i} \circ L_{m+2k}^{i+l-2} \simeq L_{m+2k+2}^{i+l} \circ S_{m+2k+2}^i$ for $l \geq 2$. 
\end{enumerate} 
then assign: 
\begin{itemize} \item $\Psi(g_{m+2k}^i) \simeq S_{m+2k}^i$, $\Psi(f_{m+2k}^i) = L_{m+2k}^i [-1] \simeq R_{m+2k}^i[1]$ \item $\Psi(t_{m+2k}^i(1)) = T_{m+2k}^i(1), \Psi(t_{m+2k}^i(2)) = T_{m+2k}^i(2)$ \item $\Psi(w_{m+2k}^i(1)) = [-1], \Psi(w_{m+2k}^i(-1)) = [1]$ 
\end{itemize} 
These functors will give a weak representation of $\textbf{FTan}_{m}$.  
\end{theorem}
\begin{proof} 
 Let us check that the relations \eqref{AFTanMovesFirst}-\eqref{TanMovesLast}, \eqref{AFTanMovesFirstTwist}-\eqref{FTanMovesLast} from Definition \ref{AFTanRelations} hold for the above choice of functors.

The Reidemeister move 0, cup-cup isotopy and cup-cap isotopy relations hold by the assumptions of the theorem,
and the cap-cap isotopy relation follows immediately from the cup-cup isotopy relation and the fact
that caps are adjoint to cups up to a shift. The cap-crossing isotopy, cup-crossing isotopy and crossing-crossing isotopy relations
follow then from the above relations and the definition of a twist. The Reidemeister move II relation 
$T_{m+2k}^i(1)T_{m+2k}^i(2)\simeq id \simeq T_{m+2k}^i(2)T_{m+2k}^i(1)$ 
follows from the fact that $S_{m+2k}^i$ are spherical functors, hence $T_{m+2k}^i(l)$ are equivalences of categories. 
The commutation relations with twists \eqref{AFTanMovesFirstTwist}-\eqref{FTanMovesLast}
hold because all exact functors commute with shifts.

The remaining less trivial relations are Reidemeister move I \eqref{AFTanReidemeister1}, 
Reidemeister move III \eqref{AFTanReidemeister3} and the pitchfork move \eqref{AFTanPitchfork}.
For simplicity of notation assume that $k=3$ and denote $\Upsilon_{m+6}^i$ by $\Upsilon_i$,
where $\Upsilon$ stands for $L$, $R$, $T(1)$ or $T(2)$.

{\bf Reidemeister move I:} $L_2T_1(1)S_2[-1] \simeq [1]$. 
We have an exact triangle
\begin{equation*}
L_2S_1R_1S_2 \to L_2S_2\to L_2T_1(1)S_2
\end{equation*}
by the definition of $T_1(1)$ and another exact triangle
\begin{equation*}
\mathrm{id}[2]\to L_2S_2\to \mathrm{id}
\end{equation*}
since $S_2$ is a strong spherical functor.
Note that the composition of maps $L_2S_1R_1S_2\to L_2S_2\to \mathrm{id}$ from these two exact triangles is in fact the adjunction counit for the pair of $L_2S_1$ and its right adjoint $R_1S_2$.
By the assumptions of the theorem, $L_2S_1$ is an equivalence, so this composition is an isomorphism.
Therefore by the octahedral axiom we have $L_2T_1(1)S_2\simeq \mathrm{id}[2]$, qed.

{\bf Pitchfork move:} $T_1(1)S_2\simeq T_2(2)S_1$.
Consider the following diagram:
\begin{equation*}
\xymatrix{
S_1R_1S_2\ar[r]\ar[d] & S_2 \ar[r]\ar[d] & T_1(1)S_2 \\
S_1[-1] \ar[r] & S_2L_2S_1[-1]\ar[r] & T_2(2)S_1
}
\end{equation*}
where the rows are exact triangles and the two vertical morphisms are induced by the isomorphisms $R_1S_2[1]\simeq \mathrm{id}$ and its dual $\mathrm{id}\simeq L_2S_1[-1]$. The diagram commutes (again because the adjunction maps for $(L_2S_1, R_1S_2)$ are compositions of adjunction maps for $(L_1,S_1,R_1)$ and $(L_2,S_2,R_2)$), therefore there is an isomorphism $T_1(1)S_2\simeq T_2(2)S_1$, qed.

{\bf Reidemeister move III:} $T_1(1)T_2(1)T_1(1) \simeq T_2(1)T_1(1)T_2(1)$. This follows from \cite{ALDGSpherical}, Theorem 1.2, since $L_iS_i$ are equivalences of categories, so the maps $L_iS_jR_jS_i\to id$ have zero cones.
\end{proof}


\subsection{Checking the tangle relations}

To apply  Theorem \ref{spherical} with $\mathcal{C}_{m+2k} = \mathcal{D}_k$, and $S_{m+2k}^i = G_{m+2k}^i$, we will need to prove that $G_{m+2n}^i: \mathcal{D}_{n-1} \rightarrow \mathcal{D}_n$ are strongly spherical functors, and check the three relations from Theorem \ref{spherical}. 

Recall that we have the inclusion of the divisor $ X_{n, i} \rightarrow U_n$, as well as the $\mathbb{P}^1$-bundle $X_{n,i} \rightarrow U_{n-1}$. 
Denote these maps by $j_{n,i}$ and $\pi_{n,i}$ respectively.
By abuse of notation we will denote $j_{n,i}^*(\mathcal{E}_k)$ simply by $\mathcal{E}_k$. The tautological sheaves $\mathcal{V}_k$ exist on
$X_{n,i}$ as well as on $U_n$, and so do their quotients.

\begin{lemma}\label{lemma-OU-of-X}
The following sheaves are isomorphic:
\begin{enumerate}
\item $\mathcal{O}_{U_n}(X_{n,i}) \simeq \mathcal{E}_{i+1}^{-1} \otimes \mathcal{E}_i$;
\item $\omega_{X_{n,i}/U_n} \simeq \mathcal{E}_{i+1}^{-1} \otimes \mathcal{E}_i\simeq \omega_{X_{n,i}/U_{n-1}}$.
\end{enumerate}
\end{lemma}
\begin{proof}
The proof of the first part is identical to the proof of Lemma 4.3 (i) in \cite{ck}. Note that we have proven that $X_{n,i}$ and $X_{n,i+1}$ intersect
transversally inside $U_n$ in Lemma \ref{lemma-transverse} here. The first isomorphism in the second part, as in part (ii) of the same Lemma in \cite{ck}, 
follows immediately from the first part and the fact that $X_{n,i}$ is a smooth divisor in $U_n$, so 
 $\omega_{X_{n,i}/U_n}=j_{n,i}^!\mathcal{O}_{U_n}[1]\simeq j_{n,i}^*\mathcal{O}_{U_n}(X_{n,i})$.
The second isomorphism follows from the canonical isomorphism $\omega_{\mathbb{P}(V)}\cong \mathcal{E}\otimes (\mathcal{V}/\mathcal{E})^{-1}$,
where $V$ is a two-dimensional space, $\mathcal{E}$ is the tautological line bundle on $\mathbb{P}(V)$, and $\mathcal{V}$ is a constant vector bundle with
fiber $V$.
\end{proof}

\begin{lemma}\label{lemma-F-G-shifted-adjoints}
We have $(G_{m+2n}^i)^R \simeq F_{m+2n}^i [-1]$ and $(G_{m+2n}^i)^L \simeq F_{m+2n}^i [1]$. 
\end{lemma}
\begin{proof}
As in the proof of Lemma 4.4 in \cite{ck}, this follows from a direct computation of the Fourier-Mukai kernels,
using the second part of Lemma \ref{lemma-OU-of-X} here.
\end{proof}

\begin{lemma}\label{lemma-FG-splitting} 
We have $F_{m+2n}^i \circ G_{m+2n}^i \simeq \mathrm{id}[-1] \oplus \mathrm{id}[1]$. 
\end{lemma} 
\begin{proof} 
This is again a direct computation. The functor $G_{m+2n}^i$ can be expressed as $(j_{n,i})_*(\mathcal{E}_i\otimes \pi_{n,i}^*(-))$. Then for its
right adjoint $(G_{m+2n}^i)^R$, which by  Lemma \ref{lemma-F-G-shifted-adjoints} is isomorphic to $ F_{m+2n}^i [-1]$, we have
$(G_{m+2n}^i)^R\simeq (\pi_{n,i})_*(\mathcal{E}_i^{-1}\otimes j_{n,i}^!(-))$. Since $j_{n,i}:X_{n,i}\to U_n$ is an embedding of a smooth divisor, we have 
$j_{n,i}^!(j_{n,i})_*\simeq \mathrm{id}\oplus (-)\otimes \mathcal{O}_{X_{n,i}}(X_{n,i})[-1] \simeq \mathrm{id}\oplus (-)\otimes \mathcal{E}_i\otimes \mathcal{E}_{i+1}^{-1}[-1]$.
Then
\begin{multline*}
F_{m+2n}^i \circ G_{m+2n}^i \simeq
 (\pi_{n,i})_*(\mathcal{E}_i^{-1}\otimes j_{n,i}^!((j_{n,i})_*(\mathcal{E}_i\otimes \pi_{n,i}^*(-))))[1] \simeq \\
\simeq
(\pi_{n,i})_* \pi_{n,i}^*(-)[1] \oplus
(\pi_{n,i})_* (\mathcal{E}_i\otimes\mathcal{E}_{i+1}^{-1}\otimes \pi_{n,i}^*(-))
\simeq
\mathrm{id}[1] \oplus \mathrm{id}[-1]
\end{multline*}
since $\pi_{n,i}$ is a Fano fibration, so $(\pi_{n,i})_* \pi_{n,i}^*\simeq \mathrm{id}\simeq (\pi_{n,i})_* \pi_{n,i}^!$, and 
$\mathcal{E}_i\otimes\mathcal{E}_{i+1}^{-1}\simeq \omega_{X_{n,i}/U_{n-1}}$
while $X_{n,i}$ has dimension $1$ over $U_{n-1}$, so 
$(\pi_{n,i})_* (\mathcal{E}_i\otimes\mathcal{E}_{i+1}^{-1}\otimes\pi_{n,i}^*)\simeq (\pi_{n,i})_* (\pi_{n,i}^!)[-1]\simeq\mathrm{id}[-1]$.
\end{proof}

Now we can show that the functors $G_{m+2n}^i$ satisfy the conditions of Theorem \ref{spherical}.

\begin{proposition}\label{prop-G-spherical}
The functors $G_{m+2n}^i: \mathcal{D}_{n-1} \rightarrow \mathcal{D}_n$ are spherical.
\end{proposition}
\begin{proof}
By definition, the functor $G_{m+2n}^i$ differs from the Fourier-Mukai functor with kernel $\mathcal{O}_{X_{n,i}}\in D^b(\mathrm{Coh}(U_{n-1}\times U_n))$
by tensoring with a line bundle, so the two functors are spherical simultaneously. By \cite{ALFibrations}, Theorem 4.2 the latter functor is spherical if 
for any $p\in U_{n-1}$ two conditions hold: first, $H^i(\Lambda^j\mathcal{N}|_{l_p})=0$ unless $i=j=0$ or $i=j=1$; second,
$(\omega_{X_{n,i}/U_n})|_{l_p}\simeq \omega_{l_p}$. Here $l_p\simeq \mathbb{P}^1$ is the fiber over $p$, and $\mathcal{N}$ is the normal bundle
of $X_{n,i}$ in $U_n$. Since $X_{n,i}\subset U_n$ is a divisor, we have $\mathcal{N}\simeq \mathcal{O}_{U_n}(X_{n,i})$, which by Lemma
\ref{lemma-OU-of-X} is isomorphic to $\mathcal{E}_{i+1}^{-1} \otimes \mathcal{E}_i$, and since $\mathcal{E}_i|_{l_p}\simeq \mathcal{O}(-1)$ and 
$\mathcal{E}_{i+1}|_{l_p}\simeq \mathcal{O}(1)$, we have $\mathcal{N}|_{l_p}\simeq \mathcal{O}(-2)$ and the first condition holds.
Then, by Lemma \ref{lemma-OU-of-X} we have $\omega_{X_{n,i}/U_n}\simeq \mathcal{E}_{i+1}^{-1} \otimes \mathcal{E}_i$
and again, since $\mathcal{E}_{i+1}^{-1} \otimes \mathcal{E}_i|_{l_p}\simeq \mathcal{O}(-2)$, the second condition holds as well.
\end{proof}

\begin{corollary} The functors $G_{m+2n}^i: \mathcal{D}_{n-1} \rightarrow \mathcal{D}_n$ are strongly spherical. \end{corollary} 
\begin{proof} 
By  Lemmas \ref{lemma-F-G-shifted-adjoints} and \ref{lemma-FG-splitting}, 
\begin{align*}
 (G_{m+2n}^i)^R G_{m+2n}^i \simeq  F_{m+2n}^i G_{m+2n}^i[-1] \simeq \mathrm{id}[-2]\oplus \mathrm{id}.
\end{align*}
The kernel of this Fourier-Mukai transform is isomorphic to  $\mathcal{O}_\Delta[-2]\oplus \mathcal{O}_\Delta$, where $\Delta\subset U_n\times U_n$ is the diagonal.
Since $\mathcal{O}_\Delta$ is a sheaf on a smooth algebraic variety, we have $\mathrm{Hom}( \mathcal{O}_\Delta,  \mathcal{O}_\Delta[-2])=0$, so the
adjunction unit  $\mathrm{id}\to (G_{m+2n}^i)^R G_{m+2n}^i$ must be a multiple of the embedding
$\mathrm{id}\xrightarrow{0\oplus \mathrm{id}}  \mathrm{id}[-2]\oplus \mathrm{id}$. This map is non-zero, since we can multiply it by $G_{m+2n}^i$
to get the map $G_{m+2n}^i \to G_{m+2n}^i (G_{m+2n}^i)^R G_{m+2n}^i$ that composes to identity with the map
$G_{m+2n}^i (G_{m+2n}^i)^R G_{m+2n}^i \to G_{m+2n}^i$ induced by the adjunction counit. Therefore, the cone of the adjunction unit
is isomorphic to $\mathrm{id}[-2]$, which proves the assertion.
\end{proof}

The following two propositions are duplicates of Propositions 5.6 and 5.16 in \cite{ck}, and the proofs from \cite{ck}, which are direct computations
with Fourier-Mukai kernels, work in our case verbatim,
except that we need to use our Corollary \ref{transverse2} instead of Corollary 5.4 from \cite{ck} for a certain transversality statement.

\begin{proposition} \label{circle} 
$F_{m+2n}^{i} \circ G_{m+2n}^{i+1} \simeq \mathrm{id} \simeq F_{m+2n}^{i+1} \circ G_{m+2n}^i$. 
\end{proposition} 

\begin{proposition} The following relations hold: \begin{enumerate} \item $G_{m+2k+2}^{i+l} \circ G_{m+2k}^i \simeq G_{m+2k+2}^i G_{m+2k}^{i+l-2}$ for $l \geq 2$; \item $G_{m+2k}^{i+l-2} \circ F_{m+2k}^{i} \simeq F_{m+2k+2}^{i} G_{m+2k+2}^{i+l}$, $G_{m+2k}^i \circ F_{m+2k}^{i+l-2} \simeq F_{m+2k+2}^{i+l} G_{m+2k+2}^i$ for $l \geq 2$.  \end{enumerate}  \end{proposition}

Now we have verified the conditions of Theorem \ref{spherical}, so we introduce the twists $T^i_{m+2n}(l)$ and construct a weak representation of $\textbf{FTan}_m$.

\begin{definition} Define the functors $T_{m+2n}^i(1)$ and $T_{m+2n}^i(2)$ via the distinguished triangles: $$G_{m+2n}^i (G_{m+2n}^i)^R \rightarrow \mathrm{id} \rightarrow T_{m+2n}^i(1), \qquad T_{m+2n}^i(2) \rightarrow \mathrm{id} \rightarrow G_{m+2n}^i (G_{m+2n}^i)^L$$ \end{definition}

\begin{theorem} The assignments $$\Psi(g_{m+2n}^i) = G_{m+2n}^i, \Psi(f_{m+2n}^i) = F_{m+2n}^i$$ $$\Psi(t_{m+2n}^i(1)) = T_{m+2n}^i(1), \Psi(t_{m+2n}^i(2)) = T_{m+2n}^i(2)$$ $$\Psi(w_{m+2n}^i(1)) = [-1], \Psi(w_{m+2n}^i(-1)) = [1]$$ give rise to a weak representation of $\textbf{FTan}_m$ using the categories $\mathcal{D}_k$. \qed \end{theorem}


\subsection{Functors $\Psi(\alpha): \mathcal{D}_p \rightarrow \mathcal{D}_q$ indexed by affine tangles} At this point, we have constructed a functor $\Psi(\alpha): \mathcal{D}_p \rightarrow \mathcal{D}_q$ for each framed linear $(m+2p, m+2q)$-tangle $\alpha$. To extend this construction to framed affine tangles, it suffices to construct a functor $\Psi(s_{m+2n}^{m+2n}): \mathcal{D}_n \rightarrow \mathcal{D}_n$ satisfying the relations in Lemma \ref{affrelns}. Define 
$S_{m+2n}^{m+2n}(\mathcal{F}) = \mathcal{F} \otimes \mathcal{E}_{m+2n}^{-1}$, and let $\Psi(s_{m+2n}^{m+2n}) := S_{m+2n}^{m+2n}$. The relations that we must check are the following:

\begin{proposition} The following identities hold, where $1 \leq i \leq m+2n-2, 1 \leq p \leq 2$: 
\begin{enumerate} 
\item $S_{m+2n-2}^{m+2n-2} \circ F_{m+2n}^i \simeq F_{m+2n}^i \circ S_{m+2n}^{m+2n}$; 
\item $ S_{m+2n}^{m+2n} \circ G_{m+2n}^i \simeq G_{m+2n}^i \circ S_{m+2n-2}^{m+2n-2}$; 
\item $S_{m+2n}^{m+2n} \circ T_{m+2n}^i(p) \simeq T_{m+2n}^i(p) \circ S_{m+2n}^{m+2n}$; 
\item $F_{m+2n}^{m+2n-1} \circ S_{m+2n}^{m+2n} \circ T_{m+2n}^{m+2n-1}(2) \circ S_{m+2n}^{m+2n} \circ T_{m+2n}^{m+2n-1}(2) \simeq F_{m+2n}^{m+2n-1}$; 
\item $S_{m+2n}^{m+2n} \circ T_{m+2n}^{m+2n-1}(2) \circ S_{m+2n}^{m+2n} \circ T_{m+2n}^{m+2n-1}(2) \circ G_{m+2n}^{m+2n-1} \simeq G_{m+2n}^{m+2n-1}$; 
\item $T_{m+2n}^{m+2n-1}(2) \circ S_{m+2n}^{m+2n} \circ T_{m+2n}^{m+2n-1}(2) \circ S_{m+2n}^{m+2n} 
\circ T_{m+2n}^{m+2n-1}(2) \simeq \\
\simeq S_{m+2n}^{m+2n} \circ T_{m+2n}^{m+2n-1}(2) \circ S_{m+2n}^{m+2n} 
\circ T_{m+2n}^{m+2n-1}(2) \circ T_{m+2n}^{m+2n-1}(2)$.
\end{enumerate} 
\end{proposition} 
\begin{proof}
The sheaf $\mathcal{E}_{m+2n}$ is constant on the fibers of $\pi_{n,i}$ for $i<m+2n-1$, thus tensoring with $\mathcal{E}_{m+2n}^{-1}$
commutes with all parts of $G_{m+2n}^i$ and $F_{m+2n}^i$, so the first three statements follow immediately. 
We will prove the fourth statement by direct computation, and the fifth statement can be proven similarly.
The sixth statement follows from the fourth, the fifth, and the exact triangle
$T_{m+2n}^{m+2n-1}(2) \to \mathrm{id} \to G_{m+2n}^{m+2n-1}\circ F_{m+2n}^{m+2n-1}[1]$.

To save space, let us skip the indices when there is no ambiguity within the current proof: denote $G_{m+2n}^{m+2n-1}$ by $G$, 
$F_{m+2n}^{m+2n-1}$ by $F$, $T_{m+2n}^{m+2n-1}(l)$ by $T(l)$, $S_{m+2n}^{m+2n}$ by $S$, 
$X_{n,m+2n-1}$ by $X$, $j_{n,m+2n-1}$ by $j$, $\pi_{n,m+2n-1}$ by $\pi$. Recall that we use the same notation for $\mathcal{E}_i$ and
$j^*\mathcal{E}_i$, so we can say that tensor multiplication by $\mathcal{E}_i$ commutes with the functors $j^*$ and $j_*$.

By definition, $T(2)=\{\mathrm{id}\to GG^L\}[-1]$, and by Lemma \ref{lemma-F-G-shifted-adjoints} we have $G^L\simeq F[1]$, so we can write
\begin{align*}
FST(2)S \simeq 
\{\pi_*(\mathcal{E}_{m+2n}^{-3}\otimes j^*(-)) \to 
\pi_*(\mathcal{E}_{m+2n}^{-2}\otimes \mathcal{E}_{m+2n-1}\otimes j^*j_*\pi^*\pi_*(\mathcal{E}_{m+2n}^{-2}\otimes j^*(-)))[1]\}[-1].
\end{align*}
The adjunction map in the cone filters through the map $\pi_*(\mathcal{E}_{m+2n}^{-3}\otimes j^*(-))\to \pi_*(\mathcal{E}_{m+2n}^{-3}\otimes j^*j_*j^*(-))$.
Recall that $j^*j_*\simeq \mathrm{id}\oplus \mathcal{E}_{m+2n}^{-1}\otimes \mathcal{E}_{m+2n-1}\otimes (-)[1]$, and the adjunction morphism
$j^*\to j^*j_*j^*\simeq j^* \oplus  \mathcal{E}_{m+2n}\otimes \mathcal{E}_{m+2n-1}^{-1}\otimes j^*$ has identity for the first component 
$j^*\to j^*$. Observe that $\pi_*(\mathcal{E}_{m+2n}^{-1}\otimes \pi^*(-))\simeq 0$ since the restriction of $\mathcal{E}_{m+2n}^{-1}$ on
a fiber of $\pi$ is isomorphic to $\mathcal{O}(-1)$. Therefore, we can further evaluate $FST(2)S$ as the cone
\begin{align*}
\{\pi_*(\mathcal{E}_{m+2n}^{-3}\otimes j^*(-)) \to 
\pi_*(\mathcal{E}_{m+2n}^{-2}\otimes \mathcal{E}_{m+2n-1}\otimes \pi^*\pi_*(\mathcal{E}_{m+2n}^{-2}\otimes j^*(-)))[1]\}[-1]
\end{align*}
where the map is induced by the adjunction map 
$\mathrm{id}\to \pi^!\pi_*\simeq \mathcal{E}_{m+2n}^{-1}\otimes \mathcal{E}_{m+2n-1}\otimes \pi^*\pi_*[1]$.
By projection formula, that turns into
\begin{align}\label{equation-FST2S-intermediate}
\{\pi_*(\mathcal{E}_{m+2n}^{-3}\otimes j^*(-)) \to 
\pi_*(\mathcal{E}_{m+2n}^{-2}\otimes \mathcal{E}_{m+2n-1})\otimes \pi_*(\mathcal{E}_{m+2n}^{-2}\otimes j^*(-)))[1]\}[-1].
\end{align}
By Grothendieck-Serre duality, since $\omega_{X/U_{n-1}}\simeq \mathcal{E}_{m+2n}^{-1}\otimes \mathcal{E}_{m+2n-1}$, we have
$\pi_*(\mathcal{E}_{m+2n}^{-2}\otimes \mathcal{E}_{m+2n-1})\simeq (\pi_*\mathcal{E}_{m+2n})^\vee[-1]$. Observe that
 $\pi$ is by construction the projectivization of $\mathcal{V}_{m+2n}/\mathcal{V}_{m+2n-2}$, whereas
$\mathcal{E}_{m+2n}=\mathcal{V}_{m+2n}/\mathcal{V}_{m+2n-1}$ is the fiberwise $\mathcal{O}(1)$, so 
$\pi_*\mathcal{E}_{m+2n}\simeq \mathcal{V}_{m+2n}/\mathcal{V}_{m+2n-2}$. We can now rewrite \eqref{equation-FST2S-intermediate}
as follows, pulling $\pi_*$ and $j^*$ out of the cone:
\begin{align*}
\pi_* \otimes
\left\{ \mathcal{E}_{m+2n}^{-3} \to ( \mathcal{V}_{m+2n}/\mathcal{V}_{m+2n-2})^\vee \otimes \mathcal{E}_{m+2n}^{-2} \right\}
\otimes j^* (-)[-1].
\end{align*}
The map within the cone is given by $\iota\otimes\mathrm{id}$, where $\iota: (\mathcal{E}_{m+2n})^\vee\to (\mathcal{V}_{m+2n}/\mathcal{V}_{m+2n-2})^\vee$
is dual to the projection $\mathcal{V}_{m+2n}/\mathcal{V}_{m+2n-2}\to \mathcal{E}_{m+2n}$, and $\mathrm{id}:\mathcal{E}_{m+2n}^{-2} \to \mathcal{E}_{m+2n}^{-2}$. From the short exact sequence
\begin{align*}
0 \to  (\mathcal{E}_{m+2n})^\vee\to (\mathcal{V}_{m+2n}/\mathcal{V}_{m+2n-2})^\vee \to (\mathcal{E}_{m+2n-1})^\vee \to 0
\end{align*}
we see that the cone is isomorphic to $\pi_*(\mathcal{E}_{m+2n}^{-2}\otimes \mathcal{E}_{m+2n-1}^{-1}\otimes j^*(-))[-1]$.
Now, the sheaf $\mathcal{E}_{m+2n}\otimes \mathcal{E}_{m+2n-1}\cong \Lambda^2(\mathcal{V}_{m+2n}/\mathcal{V}_{m+2n-2})$
and hence by Lemma \ref{lemma-ker-x-trivial-on-X} below it is trivial on $X$, so we have proven
\begin{align*}
FST(2)S \simeq \pi_*(\mathcal{E}_{m+2n}^{-1}\otimes j^*(-))[-1] \simeq F[-1].
\end{align*}
Since $G$ is spherical, we have $G^LT(1)[-1]\simeq G^R$, and since by Lemma \ref{lemma-F-G-shifted-adjoints} we know that
$G^L\simeq F[1]$ and $G^R\simeq F[-1]$, it follows that $FT(1)\simeq F[-1]$. Then $FST(2)S\simeq F[-1]$ implies $FST(2)S\simeq FT(1)$, which concludes the proof.
\end{proof}

It remains to prove the following technical lemma:
\begin{lemma}\label{lemma-ker-x-trivial-on-X}
The vector bundle $\mathcal{V}_{m+2n}/\mathcal{V}_{m+2n-2}$ on $X_{n,m+2n-1}$ is trivial.
\end{lemma}
\begin{proof}
Recall that by definition any point of $X_{n,m+2n-1}$ consists of the data of a full flag $0\subset V_1\subset\ldots\subset V_{m+2n}=V$ in a fixed
$m+2n$-dimensional space $V$, and a nilpotent element $x\in S_n$ such that $xV_i\subset V_{i-1}$ and moreover $xV_{m+2n}\subset V_{m+2n-2}$,
so $\mathrm{rk}\, x \leq m+2n-2$. From the definition of $S_n$ we see that $x$ must then satisfy $a_1=a_{m+2n}=b_1=b_{m+2n}=0$, 
and $\mathrm{ker}\, x = \langle e_1, f_1 \rangle$ is the same subspace $W\subset V$ for all points in $X_{n,m+2n-1}$. Then the map
$W\to V = V_{m+2n} \to V_{m+2n}/V_{m+2n-2}$ is an isomorphism for all points in $X_{n,m+2n-1}$ which produces an
isomorphism $W\otimes \mathcal{O}_{X_{n,m+2n-1}}\simeq \mathcal{V}_{m+2n}/\mathcal{V}_{m+2n-2}$.
\end{proof}

We have established the following theorem:
\begin{theorem} \label{afftan} The assignments \begin{align*} \Psi(g_{m+2n}^i) &= G_{m+2n}^i, \Psi(f_{m+2n}^i) = F_{m+2n}^i \\ \Psi(t_{m+2n}^i(1)) &= T_{m+2n}^i(1), \Psi(t_{m+2n}^i(2)) = T_{m+2n}^i(2) \\ \Psi(w_{m+2n}^i(1)) &= [-1], \Psi(w_{m+2n}^i(-1)) = [1] \\ \Psi(s_{m+2n}^{m+2n}) &= S_{m+2n}^{m+2n} \end{align*} give rise to a weak representation of $\textbf{AFTan}_m$ using the categories $\mathcal{D}_k$. 
\qed \end{theorem}


\section{The exotic $t$-structure on $\mathcal{D}_n$}

First we recall that the construction of the exotic $t$-structure on $\mathcal{D}_n$ from \cite{bm} is given by the following. Let $\mathbb{B}_{aff}$ denotes the braid group attached to the affine Weyl group $W_{aff} = W \ltimes \Lambda$, where $W$ is the Weyl group of $\mathfrak{g}=\mathfrak{sl}_{m+2n}$, and $\Lambda$ is the weight lattice. Let $\mathbb{B}_{aff}^{\text{Cox}} \subset \mathbb{B}_{aff}$ denote the braid group attached to the $W_{aff}^{\text{Cox}} = W \ltimes Q$ where $Q$ is the root lattice. Denote by $\mathbb{B}_{aff}^+ \subset \mathbb{B}_{aff}^{\text{Cox}}$ the semigroup generated by the lifts of the simple reflections $\tilde{s}_{\alpha}$ in the Coxeter group $W_{aff}^{\text{Cox}}$.

Using Bezrukavnikov and Mirkovic's construction (see Sections $1.1.1$ and $1.3.2$ of \cite{bm}), there exists a weak action of the affine braid group $\mathbb{B}_{aff}$ on $\mathcal{D}_n$ (i.e. for every $b \in \mathbb{B}_{aff}$, there exists a functor $\Psi(b): \mathcal{D}_n \rightarrow \mathcal{D}_n$, such that $\Psi(b_1 b_2) \simeq \Psi(b_1) \circ \Psi(b_2)$). This action is related to that from the previous section using the following result.

\begin{lemma} $\mathbb{B}_{aff}$ can be identified with the group of all bijective $(m+2n, m+2n)$ affine tangles (i.e. where each strand connects a point in the inner circle with a point in the outer circle). Under this identification, the action of $\mathbb{B}_{aff}$ on $\mathcal{D}_n$ coincides with the action coming from Theorem \ref{afftan}. \qed \end{lemma} 
\begin{proof}
The affine braid group with $m+2n$ strands is generated by linear crossings $t_{m+2n}^i(2)$, $1\leq i \leq m+2n-1$, and the braid $(s_{m+2n}^{m+2n})^{-1}$, while the monoid generated by the same elements can be identified with
 $\mathbb{B}_{aff}^+$.
We can use \cite{ALFourierMukai}, Corollary 4.5 to compute a Fourier-Mukai kernel for $T_{m+2n}^i(2)$ and see that it coincides with the Fourier-Mukai
kernel for the corresponding braid group element given in \cite{bm}. Furthermore, the functor $S_{m+2n}^{m+2n}$ from the previous section acts by tensor multiplication by the same line bundle as it should according to \cite{bm}. 
Therefore the action defined here and the action from \cite{bm} are the same up to isomorphisms of Fourier-Mukai kernels.
\end{proof}

Following Bezrukavnikov and Mirkovic (see section $1.5$ of \cite{bm}), the exotic $t$-structure on $\mathcal{D}_n$ is defined as follows:
\begin{align*} \mathcal{D}_n^{\geq 0} &= \{ \mathcal{F} \  | \  R \Gamma(\Psi(b^{-1})\mathcal{F}) \in D^{\geq 0}(\text{Vect} ) \  \forall \  b \in \mathbb{B}_{aff}^+ \} \\ \mathcal{D}_n^{\leq 0} &= \{ \mathcal{F} \  | \  R \Gamma(\Psi(b) \mathcal{F}) \in D^{\leq 0}(\text{Vect} ) \  \forall \  b \in \mathbb{B}_{aff}^+ \} \end{align*}

By definition, the functors that correspond to positive braids are left $t$-exact, and the functors that correspond to negative braids are right $t$-exact. In particular, the functor $R_n$ that corresponds to the braid $r_n$ that is both positive and negative (since it has no crossings) is $t$-exact.

\begin{proposition} \label{texact} The functor $G_{m+2n}^i: \mathcal{D}_{n-1} \rightarrow \mathcal{D}_{n}$ is $t$-exact with respect to the exotic $t$-structures on the two categories. 
\end{proposition} 

To prove this, we will need the following two lemmas:

\begin{lemma} \label{move} Given $b \in \mathbb{B}'^+_{aff}$ considered as a bijective $(m+2n-2, m+2n-2)$-tangle, there exists bijective $(m+2n-2,m+2n-2)$ tangles $\epsilon_i(b), \eta_i(b) \in \mathbb{B}^+_{aff}$ such that $b \circ f_{m+2n}^i = f_{m+2n}^i \circ \epsilon_i(b), b^{-1} \circ f_{m+2n}^{i} = f_{m+2n}^i \circ \eta_i(b)^{-1}$. \end{lemma} \begin{proof} Using the cap-crossing isotopy relation (9), we may define $\epsilon_i(t_{m+2n-2}^j)=\eta_i(t_{m+2n-2}^j)=t_{m+2n}^{j+2}$ if $j \geq i$, and $\epsilon_i(t_{m+2n-2}^j)=\eta_i(t_{m+2n-2}^j)=t_{m+2n}^{j}$ if $j \leq i-2$. Also define $\epsilon_i(t_{m+2n-2}^i)=\eta_i(t_{m+2n-2}^i)$ to be the tangle with a strand connecting $(\zeta_k,0)$ to $(2 \zeta_k,0)$ for $k \neq i-1,i$, and a strand connecting $(\zeta_{i-1},0)$ to $(2 \zeta_{i+2}, 0)$ that passes beneath a strand connecting $(\zeta_{i+2},0)$ to $(2 \zeta_{i-1},0)$. It is straightforward to check that with this definition, $\epsilon_i(t_{m+2n-2}^j)=\eta_i(t_{m+2n-2}^j) \in \mathbb{B}_{aff}^+$.

Given $b \in \mathbb{B}'^+_{aff}$, choose a decomposition $b =  t_{m+2n-2}^{i_1}(1) \circ t_{m+2n-2}^{i_2}(1) \circ \cdots \circ t_{m+2n-2}^{i_k}(1)$; clearly $\epsilon_i(b) = \epsilon_i(t_{m+2n-2}^{i_1}(1)) \circ \cdots \circ \epsilon_i(t_{m+2n-2}^{i_k}(1))$ and $\eta_i(b) = \eta_i(t_{m+2n-2}^{i_k}(1)) \circ \cdots \circ \eta_i(t_{m+2n-2}^{i_1}(1))$ satisfy the required condition. \end{proof}

\begin{lemma} \label{vanish} Given $\mathcal{F} \in \mathcal{D}_{n}^{\geq 0}$ and $\mathcal{G} \in \mathcal{D}_{n}^{\leq 0}$, we have $R\Gamma((G_{m+2n}^{m+2n-1})^R \mathcal{F}) \in \mathcal{D}^{\geq 0}(\text{Vect})$ and $R\Gamma((G_{m+2n}^{1})^L \mathcal{G}) \in \mathcal{D}^{\leq 0}(\text{Vect})$. \end{lemma} 
\begin{proof} 
Let $\mathcal{F} \in \mathcal{D}_{n}^{\geq 0}$.
The chain of isomorphisms
$$(G_{m+2n}^{m+2n-1})^L(\mathcal{E}_{m+2n}[-1])\simeq F_{m+2n}^{m+2n-1} \mathcal{E}_{m+2n}\simeq (\pi_{n,m+2n-1})_*(j_{n,m+2n-1}^* \mathcal{O}_{U_n})\simeq\mathcal{O}_{U_{n-1}}$$
implies the following:
\begin{align*} 
\text{R}\Gamma((G_{m+2n}^{m+2n-1})^R \mathcal{F}) &\simeq \text{RHom}((G_{m+2n}^{m+2n-1})^L\mathcal{E}_{m+2n}[-1],(G_{m+2n}^{m+2n-1})^R \mathcal{F}) \\ &\simeq \text{RHom}(\mathcal{E}_{m+2n}[-1],G_{m+2n}^{m+2n-1}(G_{m+2n}^{m+2n-1})^R \mathcal{F}) \\ &\simeq \text{R}\Gamma(\mathcal{E}_{m+2n}^{-1}[1] \otimes G_{m+2n}^{m+2n-1}(G_{m+2n}^{m+2n-1})^R \mathcal{F}). 
\end{align*} 
We will prove that $\mathcal{E}_{m+2n}^{-1}[1] \otimes G_{m+2n}^{m+2n-1}(G_{m+2n}^{m+2n-1})^R \mathcal{F}\in\mathcal{D}_n^{\geq 0}$, which will imply the first statement of the lemma.
Since $G_{m+2n}^{m+2n-1}$ is strongly spherical and $(G_{m+2n}^{m+2n-1})^L\simeq (G_{m+2n}^{m+2n-1})^R[2]$, it suffices to prove that $\mathcal{E}_{m+2n}^{-1} \otimes G_{m+2n}^{m+2n-1}(G_{m+2n}^{m+2n-1})^L \mathcal{F}  \simeq \mathcal{E}_{m+2n}^{-1} \otimes G_{m+2n}^{m+2n-1}(G_{m+2n}^{m+2n-1})^R \mathcal{F}[2] \in \mathcal{D}_{n}^{\geq -1}$. We have a distinguished triangle $T_{m+2n}^{m+2n-1}(2) \mathcal{F} \rightarrow \mathcal{F} \rightarrow G_{m+2n}^{m+2n-1} (G_{m+2n}^{m+2n-1})^L \mathcal{F}$; and hence a distinguished triangle 
$$\mathcal{E}_{m+2n}^{-1} \otimes T_{m+2n}^{m+2n-1}(2) \mathcal{F} \rightarrow \mathcal{E}_{m+2n}^{-1} \otimes\mathcal{F} \rightarrow \mathcal{E}_{m+2n}^{-1} \otimes G_{m+2n}^{m+2n-1} (G_{m+2n}^{m+2n-1})^L \mathcal{F}.$$ 
The functor $\mathcal{E}_{m+2n}^{-1} \otimes (-)$ corresponds to the braid $s_{m+2n}^{m+2n}$. 
Then the functors $\mathcal{E}_{m+2n}^{-1} \otimes (-)$ and $\mathcal{E}_{m+2n}^{-1} \otimes T_{m+2n}^{m+2n-1}(2)$ are left exact since they both correspond to negative braids. Thus the objects $\mathcal{E}_{m+2n}^{-1} \otimes T_{m+2n}^{m+2n-1}(2) \mathcal{F}$ and $\mathcal{E}_{m+2n}^{-1} \otimes \mathcal{F}$ are in $\mathcal{D}_{n}^{\geq 0}$, and using the long exact sequence of cohomology we obtain that $\mathcal{E}_{m+2n}^{-1} \otimes G_{m+2n}^{m+2n-1}(G_{m+2n}^{m+2n-1})^L \mathcal{F} \in \mathcal{D}_{n}^{\geq -1}$, as required.

The proof of the second half of the lemma follows the same logic. Let $\mathcal{G} \in \mathcal{D}_n^{\leq 0}$. First we show that $(G_{m+2n}^1)^L \mathcal{E}_1\simeq\mathcal{O}_{U_{n-1}}$. By definition, $F_{m+2n}^1 \mathcal{E}_1 \simeq (\pi_{n,1})_*(j_{n,1}^* \mathcal{E}_1 \otimes \mathcal{E}_2^{-1})$.
 Since the map $\pi_{n,1}: X_{n,1} \rightarrow U_{n-1}$ is a $\mathbb{P}^1$ fibre bundle, we have $(\pi_{n,1})_* \omega_{X_{n,1}} [\text{dim } X_{n,1}] \simeq \omega_{U_{n-1}} [\text{dim } U_{n-1}]$. Since $U_{n-1}$ is a symplectic variety, $\omega_{U_{n-1}} \simeq \mathcal{O}_{U_{n-1}}$; so $(\pi_{n,1})_* \omega_{X_{n,1}} \simeq \mathcal{O}_{U_{n-1}}[-1]$. By Lemma \ref{lemma-OU-of-X} we have $\omega_{X_{n,1}} \simeq j_{n,1}^* \mathcal{E}_1 \otimes \mathcal{E}_2^{-1}$, so $(G_{m+2n}^1)^L \mathcal{E}_1\simeq F_{m+2n}^1 \mathcal{E}_1[1]\simeq(\pi_{n,1})_* \omega_{X_{n,1}}[1]\simeq\mathcal{O}_{U_{n-1}}$.

It then follows that: 
\begin{align*} \text{R}\Gamma((G_{m+2n}^{1})^L \mathcal{G}) &\simeq  \text{RHom}((G_{m+2n}^{1})^L \mathcal{E}_{1},(G_{m+2n}^{1})^L \mathcal{G}) \\ 
&\simeq \text{RHom}(\mathcal{E}_{1},G_{m+2n}^{1}(G_{m+2n}^{1})^L \mathcal{G}) \\ 
&\simeq \text{R}\Gamma(\mathcal{E}_{1}^{-1} \otimes G_{m+2n}^{1}(G_{m+2n}^{m+2n-1})^L \mathcal{G}) 
\end{align*} 
The functor $\mathcal{E}_1^{-1} \otimes (-)$ is right exact as it corresponds to the positive braid $s_{m+2n}^1$ that leaves the last $m+2n-1$ vertices in place, and winds the first vertex counterclockwise around the circle underneath the other strands. Using the exact triangle $T_{m+2n}^{1}(2) \mathcal{F} \rightarrow \mathcal{F} \rightarrow G_{m+2n}^{1}(G_{m+2n}^{m+2n-1})^L \mathcal{F}$, we deduce that $G_{m+2n}^{1}(G_{m+2n}^{m+2n-1})^L$ is right exact since $T_{m+2n}^{1}(2)$ is right exact. Thus $\mathcal{E}_{1}^{-1} \otimes G_{m+2n}^{1}(G_{m+2n}^{m+2n-1})^L \mathcal{G} \in \mathcal{D}_n^{\leq 0}$, as required.  \end{proof}

Now we are ready to prove Proposition \ref{texact}.

\begin{proof} The functors $G_{m+2n}^i$ are conjugate by the $t$-exact, invertible functor $R_n$; thus it suffices to prove that $G_{m+2n}^1$ is left $t$-exact, and that $G_{m+2n}^{m+2n-1}$ is right $t$-exact, or equivalently that $(G_{m+2n}^1)^L$ is right $t$-exact and $(G_{m+2n}^{m+2n-1})^R$ is left $t$-exact

Let $\mathcal{F} \in \mathcal{D}_{n}^{\leq 0}$. To prove that $(G_{m+2n}^1)^L$ is right $t$-exact, we must show that $(G_{m+2n}^1)^L\mathcal{F} \in \mathcal{D}_{n-1}^{\leq 0}$, or in other words,
\begin{align}\label{eqFrightexact} R \Gamma( \Psi(b) (G_{m+2n}^1)^L \mathcal{F}) \in \mathcal{D}^{\leq 0}(\text{Vect}) \  \forall \  b \in \mathbb{B}'^+_{aff}.
\end{align}
By Lemma \ref{lemma-FG-splitting} we have $(G_{m+2n}^1)^L\simeq F_{m+2n}^1[1]$ and by Lemma \ref{move} we know that for any positive braid $b$ there is a positive braid $\epsilon_i(b)$ such that $\Psi(b)F_{m+2n}^1\simeq F_{m+2n}^1\Psi(\epsilon_i(b))$, so \eqref{eqFrightexact} is equivalent to 
\begin{align} \label{eqFrightexact2}
R\Gamma((G_{m+2n}^1)^L \Psi(\epsilon_i(b)) \mathcal{F}) \in \mathcal{D}^{\leq 0}(\text{Vect}) \  \forall \  b \in \mathbb{B}'^+_{aff}. 
\end{align} 
The braid $\epsilon_i(b)$ is positive, so $\Psi(\epsilon_i(b))\mathcal{F}\in\mathcal{D}^{\leq 0}$, and \eqref{eqFrightexact2} follows from Lemma \ref{vanish}.

The proof that $(G_{m+2n}^{m+2n-1})^R$ is left $t$-exact follows analogously from Lemmas \ref{move} and \ref{vanish}.
\end{proof}

\begin{definition} Let $\mathcal{D}^0_n$ denote the heart of the exotic $t$-structure on $\mathcal{D}_n$. \end{definition}

The following theorem is the main result of this section:

\begin{theorem} \label{irred} The functor $G_n^i$ sends irreducible objects in $\mathcal{D}^0_{n-1}$ to irreducible objects in $\mathcal{D}^0_n$. \end{theorem} \begin{proof} This follows from Proposition \ref{texact} and the Theorem in Section $4.2$ of \cite{bm}. \end{proof}


\subsection{Irreducible objects in the heart of the exotic $t$-structure on $\mathcal{D}_n$}

\begin{definition} Let an affine crossingless $(m,m+2n)$ matching be an affine $(m,m+2n)$-tangle whose vertical projection to $\mathbb{C}$ has no crossings, with the blackboard framing. Let an unlabelled affine crossingless matching $(m,m+2n)$-matching be an affine crossingless matching where the $m$ inner points are not labelled. Let $\text{Cross}(m,n)$ be the set of all unlabelled affine crossingless matchings. \end{definition}
\begin{figure}
\centering
\includegraphics[scale=0.75]{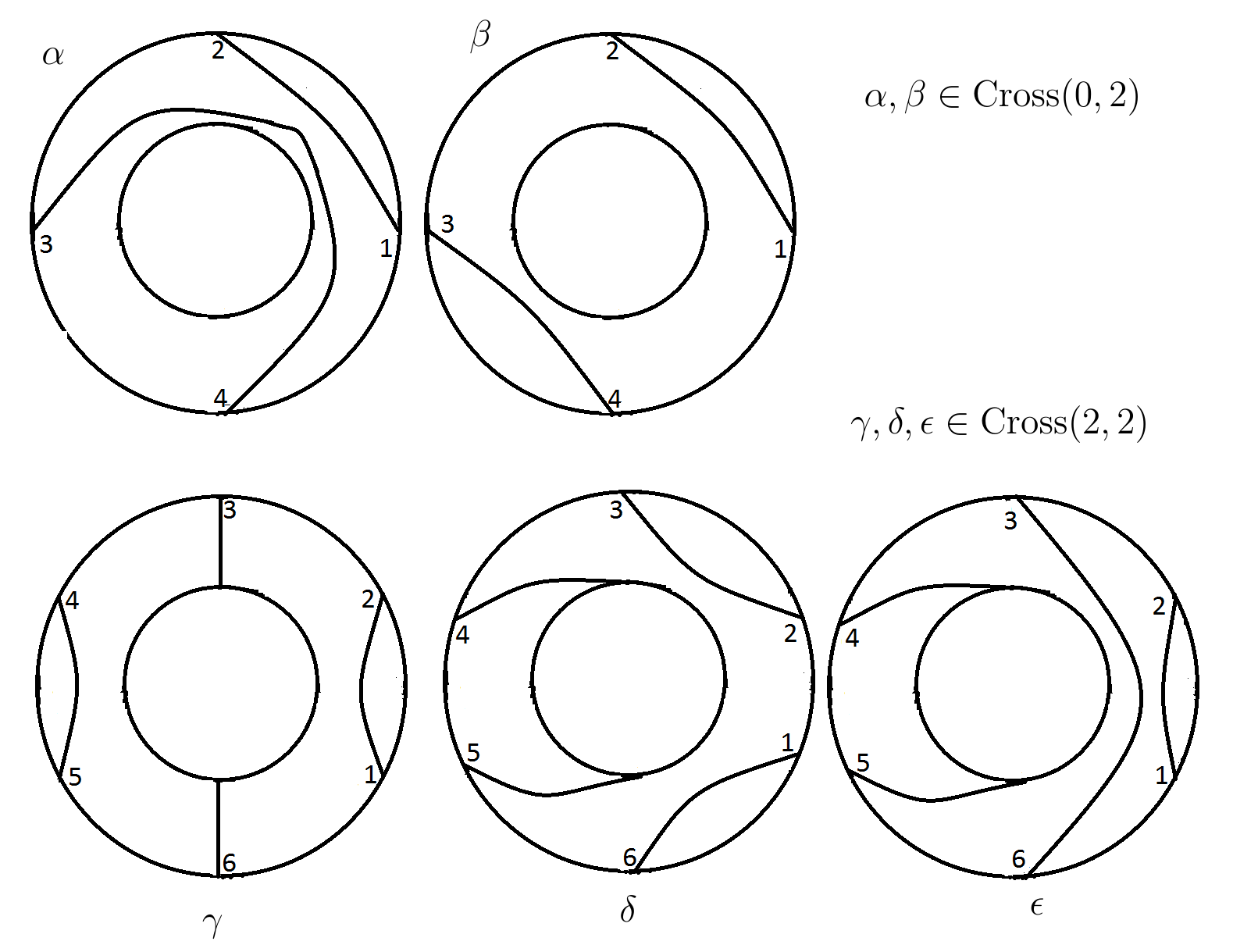}
\end{figure}

We will describe the irreducible objects in the heart of the exotic $t$-structure on $\mathcal{D}_n$ using the functors constructed in the previous section. 
\begin{lemma} We have $|\text{Cross}(m,n)|=\binom{m+2n}{n}$. \end{lemma}
\begin{proof} It suffices to construct a bijection between unlabelled affine $(m,m+2n)$ crossingless matchings and assignments of $m+n$ plus signs and $n$ minus to $m+2n$ labelled points on a circle. Given such an assignment of pluses and minuses to the points $(2,0), (2 \zeta_{m+2n}, 0), \cdots, (2 \zeta_{m+2n}^{m+2n-1}, 0)$, for each minus, move anti-clockwise around the circle and connect the minus to the first plus such that the number of pluses and minuses between these two points is equal. After connecting the $m$ remaining pluses on the outer circle to the $m$ unlabelled points on the inner circle without crossings, we have our desired unlabelled affine $(m,m+2n)$ crossingless matching. \end{proof}
\begin{lemma} Let $\widetilde{\alpha}$ be any affine $(m,m+2n)$-tangle, from which we obtain $\alpha$ by forgetting the labelling on the inner circle. Then the isomorphism class of the functor $\Psi(\widetilde{\alpha})$ depends only on $\alpha$. \end{lemma} \begin{proof} It is easy to check $\Psi(r_m)$ is isomorphic to the identity on $\mathcal{D}_0$; the result follows.\end{proof}
Thus given $\alpha \in \text{Cross}(m,n)$, we obtain a functor $\Psi(\alpha): \mathcal{D}_0 \rightarrow \mathcal{D}_n$. Let $\Psi_{\alpha}=\Psi(\alpha) \underline{v}$ denote the image of the $1$-dimensional vector space $\underline{v}$ in $\mathcal{D}_0 \simeq D^b(\text{Vect})$ under the functor $\Psi(\alpha)$.

\begin{proposition}\label{PropIrredObjects} The irreducible objects in the heart of the exotic $t$-structure on $\mathcal{D}_n$ are precisely given by $\Psi_{\alpha}$, as $\alpha$ ranges across the $\binom{m+2n}{n}$ unlabelled affine $(m,m+2n)$ crossingless matchings. \end{proposition} 
\begin{proof} It follows by induction that every affine crossingless $(m,m+2n)$ matching can be expressed as a product $g_{m+2n}^{i_{n}} \circ \cdots \circ g_{m+4}^{i_2} \circ g_{m+2}^{i_1}$. Thus $\Psi_{\alpha} \simeq G_{m+2n}^{i_n} \circ \cdots \circ G_{m+4}^{i_2} \circ G_{m+2}^{i_1} \underline{v}$; by Proposition \ref{irred}, $\Psi_{\alpha}$ is an irreducible object in the heart of the exotic $t$-structure. It will follow from the arguments in the next section that these irreducible objects are distinct. Since $K^0(\mathcal{D}_n^{\geq 0} \cap \mathcal{D}_n^{\leq 0})\simeq K^0(\mathcal{D}_n)\simeq K^0(\text{Coh}(\mathcal{B}_{z_n}))$, and $K^0(\text{Coh}(\mathcal{B}_{z_n}))$ has rank $\binom{m+2n}{n}$, these constitute all the irreducible objects in the heart of the exotic $t$-structure on $\mathcal{D}_n$. \end{proof}


\subsection{The $\text{Ext}$ space}
The goal of this section is to describe the space 
\begin{equation}\label{equation-Ext-algebra}
\text{Ext}^{\bullet}(\displaystyle \bigoplus_{\alpha \in \text{Cross}(m,n)} \Psi_{\alpha}, \displaystyle \bigoplus_{\alpha \in \text{Cross}(m,n)} \Psi_{\alpha}).
\end{equation}
Let $\gamma$ be an affine $(m+2p,m+2q)$ tangle. 
Denote by $\breve{\gamma}$ the affine $(m+2q,m+2p)$ tangle obtained by taking the inversion of $\gamma$ with respect to the unit circle and scaling so that we get an affine tangle.
\begin{lemma} 
The right adjoint to $\Psi(\gamma)$ is $\Psi(\breve{\gamma})[p-q]$. \end{lemma} 
\begin{proof} 
From Section $2.3$, the right adjoint to $\Psi(\gamma)$ is $\Psi(\breve{\gamma})[p-q]$, when $\gamma$ is one of $g_{m+2n}^{k}$, $f_{m+2n}^k$, or $t_{m+2n}^k(l)$. The result follows, since any tangle is the composition of these tangles, and if $\gamma = \gamma_1 \circ \gamma_2$, $\breve{\gamma} = \breve{\gamma_2} \circ \breve{\gamma_1}$. 
\end{proof}
Now we will compute $\text{Ext}^{\bullet}(\Psi_{\alpha}, \Psi_{\beta})$ as a vector space, where $\alpha, \beta \in \text{Cross}(m,n)$.
\begin{align*} \text{Ext}^{\bullet}(\Psi_{\alpha}, \Psi_{\beta}) &= \text{Ext}^{\bullet}(\Psi(\alpha) \underline{v}, \Psi(\beta) \underline{v}) \\ &\cong \text{Ext}^{\bullet}(\underline{v}, \Psi(\breve{\alpha} \circ \beta)[-n] \underline{v}) \\ &\cong \Psi(\breve{\alpha} \circ \beta)[-n] \underline{v} \end{align*}
The explicit description is quite different for $m=0$ and $m>0$, since when $m=0$ the tangle $\breve{\alpha}\circ\beta$ is a link and may contain circles that go around the origin, and when $m>0$ this tangle either contains threads that go from the inner boundary to the outer boundary (and hence does not contain circles that go around the origin), or yields a zero functor.

Let $\Lambda \in D^b(\text{Vect})$ be a complex concentrated in degrees $1$ and $-1$, with dimension $1$ in those degrees, and let $\Lambda_0\in D^b(\text{Vect})$ be a complex that consists of a two-dimensional space $\mathbb{C}^2$ in degree $0$.
\begin{lemma} \label{linkgen} 
If $m>0$ then $\Psi(f_{m+2}^i \circ g_{m+2}^k) \underline{v} \cong \begin{cases} \Lambda \; \mbox{ if   } i=k \\ \underline{v} \; \mbox{ if   } |i-k|=1\mbox{ or   } \{i,k\}=\{1,m+2\} \\ 0 \; \mbox{ otherwise} \end{cases}$ 
\end{lemma} 
\begin{proof} Firstly if $i=m+2$ or $k=m+2$ and $m>0$ we can conjugate $f_{m+2}^i \circ g_{m+2}^k$ by a power of $r_m$, so it is safe to assume that $1 \leq i,k \leq m+1$. From Section $2.3$, $\Psi(g_{m+2}^k) \underline{v} \simeq  (j_{1,k})_*(\mathcal{E}_k \otimes \mathcal{O}_{X_{1,k}})\simeq (j_{1,k})_*(\mathcal{E}_k)$, since $\pi_{1,k}$ maps $X_{1,k}$ to the point $S_0 \times_{\mathfrak{sl}_m} T^* \mathcal{B}_0$. Then
\begin{multline} \label{ext} 
\Psi(f_{m+2}^i \circ g_{m+2}^k) \underline{v} \simeq \Psi(f_{m+2}^i)((j_{1,k})_*(\mathcal{E}_k))  \simeq \\
\simeq (\pi_{1,i})_*[j_{1,i}^*(j_{1,k})_*((\mathcal{E}_k) \otimes \mathcal{E}_{i+1}^{-1})]  
\simeq (\pi_{1,i})_*[j_{1,i}^*(j_{1,k})_*(\mathcal{E}_k) \otimes \mathcal{E}_{i+1}^{-1}].
\end{multline}
Since the below conditions defining $X_{1,i}$ force $x=z_1$ (the standard nilpotent of type $(m+1,1)$ acting on the Jordan basis $\{e_1, \cdots, e_{m+1}, f_1\}$), $V_l = \langle e_1, \cdots, e_l \rangle$ for $1 \leq l \leq i-1$, and $V_l = \langle e_1, \cdots, e_{l-1}, f_1 \rangle$ for $l \geq  i+1$, we have: 
\begin{align*} 
X_{1,i} &= \{ (0 \subset V_1 \subset \cdots \subset V_{i-1} \subset V_i \subset V_{i+1} \subset \cdots \subset V_{m+2}), x | \\ 
& \qquad x \in S_1,\ x V_{i+1} \subset V_{i-1},\ xV_l \subset V_{l-1} \} = \mathbb{P}(V_{i+1}/V_{i-1}).
\end{align*} 
Thus the intersection $X_{1,i} \cap X_{1,i+1}$ consists of the single point $\{(0 \subset V_1 \subset \cdots \subset V_{m+2}),z_1\}$, where $V_l = \langle e_1, \cdots, e_l \rangle$ for $1 \leq l \leq i$, and $V_l = \langle e_1, \cdots, e_{l-1}, f_1 \rangle$ for $l \geq i+1$. Moreover, we have $X_{1,i} \cap X_{1,k} = \emptyset$ if $|i-k| > 1$.

When $i=k$, since $j_{1,i}: X_{1,i}\to U_1$ is an embedding of a divisor and by Lemma \ref{lemma-OU-of-X} we have 
$\mathcal{O}_{U_1}(X_{1,i})\simeq \mathcal{E}_i\otimes \mathcal{E}_{i+1}^{-1}$,
 there is an isomorphism $j_{1,i}^*(j_{1,i})_*\simeq \mathrm{id}\oplus \mathcal{E}_i^{-1}\otimes\mathcal{E}_{i+1}[1]$.
Together with \eqref{ext}, this  implies
\begin{align*}
\Psi(f_{m+2}^i \circ g_{m+2}^i) \underline{v} \simeq 
(\pi_{1,i})_*\left(\mathcal{E}_i\otimes \mathcal{E}_{i+1}^{-1} \oplus \mathcal{O}_{X_{1,i}}[1]\right)
\simeq \mathcal{O}_{U_0}[-1] \oplus \mathcal{O}_{U_0}[1] \simeq \Lambda.
\end{align*}

If $|i-k| > 1$, since $ j_{1,i}^* (j_{1,k})_* \mathcal{F} = 0$ for any $\mathcal{F} \in \text{Coh}(X_{1,i})$ (as $X_{1,i} \cap X_{1,k}=\emptyset$), we obtain $\Psi(f_{m+2}^i \circ g_{m+2}^k) \underline{v} = 0$.

If $|i-k|=1$, we have the Reidemeister 0 move $f_{m+2}^i\circ g_{m+2}^k\sim id$, which implies the statement of the lemma.
\end{proof} 
\begin{lemma} \label{linkgenm0} 
If $m=0$ then $\Psi(f_{2}^i \circ g_{2}^j) \underline{v} \simeq 
\begin{cases} 
\Lambda \; \mbox{ if   } i=j \\ 
\Lambda_0 \; \mbox{ otherwise} 
\end{cases}$ 
\end{lemma} 
\begin{proof}
In this case we have $z_1=0$, so $S_1=\mathcal{N}$, and $U_1=T^*\mathcal{B}_1=T^*\mathbb{P}^1=T^*\mathbb{P}(V)$, where $V$ is the standard representation of $\mathfrak{sl}_2$. Consequently, $\mathcal{B}_{z_1}$ is the image of the zero section $\iota:\mathbb{P}^1\to T^*\mathbb{P}^1$. 
Thus our categories are $\mathcal{D}_0=D^b(Vect)$ and 
$\mathcal{D}_1=D^b(\mathrm{Coh}_{\mathbb{P}^1}T^*\mathbb{P}^1)$.
Furthermore, we have $X_{1,1}=pt\times \iota(\mathbb{P}^1)\subset pt\times T^*\mathbb{P}^1=U_0\times U_1$. We can find the sheaves
$\mathcal{E}_1\cong\mathcal{O}(-1)$ and $\mathcal{E}_2\cong T\mathbb{P}^1(-1)$ and the functors
$G_2^1 =(-)\otimes\iota_*\mathcal{O}(-1)$ and $\mathcal{S}_2^2 =(-)\otimes \mathcal{E}_2^{-1}\cong (-)\otimes T^*\mathbb{P}^1(1)$. Then we can compute
$G_2^2 = S_2^2 T_2^1(2)G_2^1 \simeq S_2^2 G_2^1[1] \simeq (-)\otimes \iota_*T^*\mathbb{P}^1[1]$. Therefore
\begin{align*}
\Psi(f_{2}^1 \circ g_{2}^1) \underline{v} & \simeq \mathrm{RHom}_{T^*\mathbb{P}^1}(\iota_*\mathcal{O}(-1), \iota_*\mathcal{O}(-1))\simeq
 \mathrm{RHom}_{\mathbb{P}^1}(\iota^*\iota_*\mathcal{O}, \mathcal{O})\simeq \Lambda\\
\Psi(f_{2}^2 \circ g_{2}^1) \underline{v} & \simeq \mathrm{RHom}_{T^*\mathbb{P}^1}(\iota_*\mathcal{O}(-1), \iota_*T^*\mathbb{P}^1[1] )\simeq
 \mathrm{RHom}_{\mathbb{P}^1}(\iota^*\iota_*\mathcal{O}, \mathcal{O}(-1)[1])\simeq \Lambda_0\\
\Psi(f_{2}^1 \circ g_{2}^2) \underline{v} & \simeq \mathrm{RHom}_{T^*\mathbb{P}^1}(\iota_*T^*\mathbb{P}^1[1], \iota_*\mathcal{O}(-1) )\simeq
 \mathrm{RHom}_{\mathbb{P}^1}(\iota^*\iota_*\mathcal{O}, \mathcal{O}(1)[-1])\simeq \Lambda_0\\
\Psi(f_{2}^2 \circ g_{2}^2) \underline{v} & \simeq \mathrm{RHom}_{T^*\mathbb{P}^1}( \iota_*T^*\mathbb{P}^1[1],  \iota_*T^*\mathbb{P}^1[1])\simeq
 \mathrm{RHom}_{\mathbb{P}^1}(\iota^*\iota_*\mathcal{O}, \mathcal{O})\simeq \Lambda
\end{align*}
since $\iota^*\iota_*\mathcal{O}\simeq \mathcal{O}\oplus\mathcal{O}(2)[1]$.
\end{proof}
\begin{definition} Define an $m$-link as an affine crossingless $(m,m)$-tangle, where both the $m$ inner points, and the $m$ outer points are unlabelled. If each of the $m$ points in the inner circle are joined to $m$ points in the outer circle, say that the $m$-link is ``good" , and otherwise (i.e. if there are cups and caps) say that the $m$-link is ``bad". If an $m$-link $\gamma$ is good, denote by $\omega(\gamma)$ the number of loops in $\gamma$ that don't go around the origin, and by $\omega_0(\gamma)$ the number of loops that do. If an $m$-link $\gamma$ is bad, then define $\omega(\gamma) = -1$. \end{definition} Continuing on from Example, we have that:
\begin{example} 
\begin{align*} \omega_0(\alpha, \beta) &= \omega(\alpha, \beta) = 1 \\
\omega(\gamma, \delta) &= \omega(\gamma, \epsilon) = -1,  \omega(\delta, \epsilon) = 1 \end{align*} Note that if $m=0$ the link is always good, and if the link $\gamma$ is good with $m>0$, then $\omega_0(\gamma)=0$. In the former (resp. latter) case, a good $m$-link $\gamma$ is determined up to an isotopy by the the numbers $\omega(\gamma)$ and $\omega_0(\gamma)$ (resp. the numbers $\omega(\gamma)$). \end{example} Given unlabelled affine crossingless $(m,m+2n)$ matchings $\alpha,  \beta$, we can construct an $m$-link $\check{\alpha} \circ \beta$; furthermore, any $m$-link $\gamma$ corresponds to a functor $\Psi(\gamma): \mathcal{D}_0 \rightarrow \mathcal{D}_0$.  It follows from Lemma \ref{linkgen} that: 
\begin{proposition} 
If the $m$-link $\gamma$ is bad, then $\Psi(\gamma)$ is zero. If the $m$-link $\gamma$ is good, then $\Psi(\gamma)$ is isomorphic to tensor multiplication by $\Lambda^{\otimes \omega(\gamma)}\otimes\Lambda_0^{\otimes\omega_0(\gamma)}$. 
\end{proposition} 
\begin{proof} The $m$-link $f_{m+2}^i \circ g_{m+2}^i$ is good, with $\omega(f_{m+2}^i \circ g_{m+2}^i)=1$; and $\Psi(f_{m+2}^i \circ g_{m+2}^i)$ corresponds to multiplication by $\Lambda$. If $|i-j|=1$, the $m$-link $f_{m+2}^i \circ g_{m+2}^j$ is good, with $\omega(f_{m+2}^i \circ g_{m+2}^j)=0$; and $\Psi(f_{m+2}^i \circ g_{m+2}^j)$ is the identity functor $\mathcal{D}_0 \rightarrow \mathcal{D}_0$. If $|i-j|>1$, the $m$-link $f_{m+2}^i \circ g_{m+2}^i$ is bad, and $\Psi(f_{m+2}^i \circ g_{m+2}^i) \underline{v}=0$ for any $w \in \mathcal{D}_0$. Call an $m$-link ``basic" if it is of the form $f_{m+2}^i \circ g_{m+2}^j$; then any $m$-link $\gamma$ can be written as a composition of basic $m$-links. The conclusion then follows from our knowledge of $\Psi(\gamma)$ for basic $m$-links $\gamma$, and the fact that $m$-link $\gamma$ is bad iff each expression of $\gamma$ in terms of basic $m$-links contains at least one bad basic $m$-link. \end{proof}
Note that the $\mathbb{Z}$-graded algebra $\text{Ext}^{\bullet}(\Psi_{{g}_{m+2}^i})=\Lambda[-1]$ is isomorphic to $\mathbb{C}[x]/(x^2)$, where $x$ has degree $2$. Indeed, this is the only possible algebra structure on $\Lambda[-1]$ which respects its grading. Now, to simplify the statement of the following result, let us adopt the convention that $\Lambda^{\otimes -1}$ is the zero complex. Thus:

\begin{theorem} \label{final} For any $\alpha, \beta\in \text{Cross}(m,n)$ we have an isomorphism of vector spaces: \begin{align*} \text{Ext}^{\bullet}(\Psi_{\alpha}, \Psi_{\beta}) \simeq \Lambda^{\otimes \omega(\check{\alpha} \circ \beta)}\otimes\Lambda_0^{\otimes\omega_0(\check{\alpha} \circ \beta)}[-n] \end{align*} \end{theorem}

\subsection{The $\text{Ext}$ algebra: conjecture}
We expect the isomorphism in Theorem \ref{final} to be canonical, which allows us to describe the multiplication in $\text{Ext}$ algebra \eqref{equation-Ext-algebra} explicitly. The resulting Ext algebra is an annular version of Khovanov's arc algebra; we will refer to them as ``annular arc algebras''. 

It suffices to describe the map 
\begin{align*} \text{Ext}^{\bullet}(\Psi_{\alpha}, \Psi_{\beta}) \otimes \text{Ext}^{\bullet}(\Psi_{\beta}, \Psi_{\gamma}) &\rightarrow \text{Ext}^{\bullet}(\Psi_{\alpha}, \Psi_{\gamma}) \end{align*} In the $m=0$ case, we obtain a map: \begin{align*}  \Lambda^{\omega(\check{\alpha} \circ \beta)}\otimes\Lambda_0^{\otimes\omega_0(\check{\alpha} \circ \beta)}\otimes \Lambda^{\omega(\check{\beta} \circ \gamma)}\otimes\Lambda_0^{\otimes\omega_0(\check{\beta} \circ \gamma)}[-2n] &\rightarrow \Lambda^{\omega(\check{\alpha} \circ \gamma)}\otimes\Lambda_0^{\otimes\omega_0(\check{\alpha} \circ \gamma)}[-n] \end{align*} In the $m>0$ case, we obtain a map: 
\begin{align*}  \Lambda^{\omega(\check{\alpha} \circ \beta)}\otimes \Lambda^{\omega(\check{\beta} \circ \gamma)}[-2n] &\rightarrow \Lambda^{\omega(\check{\alpha} \circ \gamma)} [-n] \end{align*} Our conjectural description of this map is as follows.  Consider a sequence of links, the first of which is a disjoint union of $\check{\alpha} \circ \beta$ and $\check{\beta} \circ \gamma$, and the last of which is $\check{\alpha} \circ {\gamma}$. Each subsequent link is obtained from the previous link by performing a ``surgery'': pick two strands between $i$ and $j$ in $\beta \circ \check{\beta}$, and replace them with two radial lines at $i$ and $j$, so that the resulting tangle has no crossings. It is clear that at each step, there is at least one way of performing a surgery operation. Depending on whether the strands between $i$ and $j$ are part of a circle (not enclosing the origin), a $0$-circle (enclosing the origin, in the $m=0$ case),  or a a line/cup/cap (in the $m>0$ case), then diagrammatically each surgery operation corresponds to one of the following. To write down a corresponding map between the $\text{Ext}$ spaces, we introduce bases $1\in\mathbb{C}$, $\{1,X \}$ in $\Lambda$, where $1$ has grading $-1$ and $X$ has grading $1$, and an arbitrary basis $\{Y_1,Y_2\}$ in $\Lambda_0$.

\textbf{\underline{Case where $m=0$}:}

\begin{itemize} \item Merging two disjoint circles into one circle: $\Lambda \otimes \Lambda[-2] \rightarrow \Lambda[-1]$,
$$1\otimes 1\mapsto 1,\quad 1\otimes X\mapsto X,\quad X\otimes 1\mapsto X,\quad X\otimes X\mapsto 0$$
\item Merging two nested circles into one circle: $\Lambda \otimes \Lambda[-2] \rightarrow \Lambda[-1]$,
$$1\otimes 1\mapsto 1,\quad 1\otimes X\mapsto -X,\quad X\otimes 1\mapsto X,\quad X\otimes X\mapsto 0$$ 
\item Splitting one circle into two disjoint circles: $\Lambda[-1] \rightarrow \Lambda \otimes \Lambda[-2]$, $1\mapsto 1\otimes X+X\otimes 1$, $X\mapsto X\otimes X$;
\item Splitting one circle into two nested circles: $\Lambda[-1] \rightarrow \Lambda \otimes \Lambda[-2]$, $1\mapsto -1\otimes X+X\otimes 1$, $X\mapsto X\otimes X$;
\item Merging a circle and a $0$-circle: $\Lambda\otimes\Lambda_0[-2]\to\Lambda_0[-1]$, 
$1\otimes Y_i\mapsto Y_i$, $X\otimes Y_i\mapsto 0$;
\item Splitting a $0$-circle into a circle and a $0$-circle: $\Lambda_0[-1]\to\Lambda\otimes\Lambda_0[-2]$,
$Y_i\mapsto X\otimes Y_i$
\item Merging two 0-circles into a circle: $$\Lambda_0 \otimes \Lambda_0 \to \Lambda, (a_1 Y_1 + a_2 Y_2) \otimes (b_1 Y_1 + b_2 Y_2) = (a_1 b_2 - a_2 b_1) X$$
\item Splitting a circle into two 0-circles: $\Lambda \rightarrow \Lambda_0 \otimes \Lambda_0$, $X \mapsto 0$, $1 \mapsto Y_1 \otimes Y_2 - Y_2 \otimes Y_1$
\end{itemize}

\textbf{\underline{Case where $m>0$}:}
\begin{itemize} \item Merging a line and a circle, into a line: $\mathbb{C} \otimes \Lambda[-1] \rightarrow \mathbb{C}$, $1\otimes 1\mapsto 1$, $1\otimes X\mapsto 0$ 
\item Splitting a line, into a line and a circle: $\mathbb{C}[-1] \rightarrow \mathbb{C} \otimes \Lambda$, $1\mapsto 1\otimes X$ \item Merging two circles into two disjoint circles (or into two nested circles); splitting one circle into two disjoint circles (or two nested circles). These are exactly the same as in the $m=0$ case
\item Merging two lines into two arcs: $\underline{\mathbb{C}} \rightarrow 0$
\item Merging two arcs into two lines: $0 \rightarrow \underline{\mathbb{C}}$
\item Merging a line and an arc, into (a different) line and an arc; Merging a circle and an arc, into an arc; splitting an arc, into a circle and an arc: $0 \rightarrow 0$
\item Merging two lines into two lines: $\underline{\mathbb{C}} \rightarrow \underline{\mathbb{C}}$ 
\end{itemize}

By iteratively applying these ``surgery'' operations (so that all the cups and caps in $\beta$ and $\check{\beta}$ are replaced by radial lines), we arrive at the desired map. 

\section{Further directions}
\subsection{Decategorification}

Denote by $V$ the $2$-dimensional representation of $\mathfrak{sl}_2$. In section $6$ (see Theorem $6.2$ and Section $6.4$) of \cite{ck}, Cautis and Kamnitzer prove that: $$ K^0(\text{Coh}(Y_{m+2n})) \simeq V^{\otimes m+2n}$$ 
Recall also that there is a map, which is compatible under composition (see Section $6.1$ of \cite{ck} for an explicit description). $$\psi: \{ (k, l) \text{-tangles} \} \rightarrow \text{Hom}_{U(\mathfrak{sl}_2)}(V^{\otimes k}, V^{\otimes l})$$
In Section 6 of \cite{ck}, it is proven that for each $(m+2p, m+2q)$-tangle $\alpha$, the functor $\widetilde{\Psi}(\alpha)$ corresponds to $\psi(\alpha)$ on the level of the Grothendieck group.  In fact, Cautis and Kamnitzer work with a $q$-deformation of this picture (using $\mathbb{C}^*$-equivariant sheaves, and representations of $U_{q}(\mathfrak{sl}_2)$); and it is also possible to introduce $\mathbb{C}^*$-equivariance in our setting.

Under the natural embedding $U_n \rightarrow Y_{m+2n}$ (constructed in Section $2.1$), we have a natural map: $$ K^0(\mathcal{D}_n^0) = K^0(\text{Coh}_{\mathcal{B}_{z_n}}(U_n)) \rightarrow K^0(\text{Coh}(Y_{m+2n})) \simeq V^{\otimes m+2n} $$ 

In future work, we will show that $K^0(\mathcal{D}_n^0)$ can be naturally identified with the $m$-weight space in $V^{\otimes m+2n}$ (which we shall denote by $V^{\otimes m+2n}_{[m]}$). After taking the image of the maps $\Psi(\alpha)$ in the Grothendieck group (where $\alpha$ is an affine tangle), we obtain: $$ \hat{\psi}: \{ (m+2k, m+2l) \text{-affine tangles} \} \rightarrow \text{Hom}(V^{\otimes m+2k}_{[m]}, V^{\otimes m+2l}_{[m]}) $$ 

We will give an explicit description of the map $\hat{\psi}$ (this almost follows Cautis and Kamnitzer's results in Section 6 of \cite{ck}), and compute the images of the irreducible objects $\Psi(\alpha)$ in $\mathcal{D}_n^0$. These give us a basis in the $(m+n,n)$ weight space in $V^{\otimes m+2n}$; we expect that this will coincide with Lusztig's canonical (or perhaps the dual canonical) basis.

\subsection{Applications to modular representation theory}

Recall, from the introduction that, Theorem $5.3.1$ from \cite{bmr} (see also Section $1.6.2$ from \cite{bm}) states that there is an equivalence: \begin{align*} D^b(\text{Coh}_{\mathcal{B}_{e, \textbf{k}}}(\widetilde{\mathfrak{g}}_{\textbf{k}})) \simeq D^b(\text{Mod}^{fg, \lambda}_{e}(U_{\textbf{k}})) \end{align*} Further, the tautological t-structure on the right hand side corresponds to the exotic t-structure on the left hand side.  Thus, by studying the irreducible objects in the heart of the exotic $t$-structure on the other side, one may derive information about irreducible objects in $\text{Mod}^{fg, \lambda}_{e}(U_{\textbf{k}})$. In the case where $e$ is a two-block nilpotent, our results give a fairly explicit description of the irreducible objects in the former category (by repeatedly applying the functors $G_{m+2n}^i$). In ongoing work, we will give combinatorial formulae for the dimensions and characters of irreducible representations lying in $\text{Mod}^{fg, \lambda}_{e}(U_{\textbf{k}})$. 

More precisely, the dimension of the modules should be related to computing the Euler characteristic of the corresponding exotic sheaves (after tensoring by a line bundle); and the characters should correspond to computing the Euler characteristic in the equivariant category (where the group acting is a maximal torus inside the centralizer of the nilpotent). Computing these Euler characteristics is related to computing the image of the irreducible objects in the Grothendieck group (the problem discussed in the previous section). 

In the last subsection, we gave a (partly conjectural) diagrammatic description of the $\text{Ext}$ algebra which governs the heart of the exotic $t$-structure; i.e. category $\text{Mod}^{fg, \lambda}_{e}(U_{\textbf{k}})$ when we work over a field $\textbf{k}$ of characteristic $p$. It is known that these categories are Koszul; thus, the category $\text{Mod}^{fg, \lambda}_{e}(U_{\textbf{k}})$ is governed by the Koszul dual of this diagram algebra. In future work, we will show that the Koszul dual is an annular version of the algebras introduced by Webster in \cite{webster}. This may be viewed as a characteristic $p$ analogue of some of the main results proven by Brundan and Stroppel in \cite{bs} (a description of the diagrammatic algebra which controls the principal block of parabolic category $\mathcal{O}$, for the parabolic with Levi sub-algebra $\mathfrak{gl}_m \oplus \mathfrak{gl}_n$ in $\mathfrak{gl}_{m+n}$), and Webster in \cite{webster} (a description of the diagrammatic algebra which controls the singular block of category $\mathcal{O}$ that is Koszul dual to the one described above). As a consequence, we would then obtain combinatorial formulae for the composition multiplicities of the irreducibles inside the indecomposable projectives in $\text{Mod}^{fg, \lambda}_{e}(U_{\textbf{k}})$. 

\subsection{A characteristic $p$ analogue of Bernstein-Frenkel-Khovanov, and annular Khovanov homology} From the discussion in the above subsections, we have constructed a map: \begin{equation} \label{reshtur} \hat{\psi}: \{ (m+2k, m+2l) \text{-affine tangles} \} \rightarrow \text{Hom}(V^{\otimes m+2k}_{[m]}, V^{\otimes m+2l}_{[m]}) \end{equation} This map is categorified by the functors $\Psi(\alpha): \mathcal{D}_p \rightarrow \mathcal{D}_q$ between categories of coherent sheaves on Springer fibers. 

Let us restrict to the case of linear tangles. In \cite{khov1} and \cite{khov2}, Khovanov and Chen construct a categorification of the invariant $\psi(\alpha): V^{\otimes m} \rightarrow V^{\otimes n}$ using categories of modules over certain diagram algebras; the functors which categorify the action of the generators $g_n^i, f_n^i$ and $t_n^i(1), (2)$ correspond to tensoring with certain (complexes of) bi-modules. The diagram algebras appearing here are similar in nature to the Ext algebras used in their paper; however, the crossingless matchings that appear in our set-up are drawn on a line (instead of a circle). The categorification constructed in \cite{khov2} has been shown to be equivalent to the one constructed earlier by Bernstein, Frenkel and Khovanov in \cite{bfk}, using parabolic blocks of category $\mathcal{O}$ (for the parabolic with Levi sub-algebra $\mathfrak{gl}_m \oplus \mathfrak{gl}_n$ in $\mathfrak{gl}_{m+n}$). 

In future work, we will give a characteristic $p$ analogue of this construction (working with affine tangles instead of linear tangles). The categories $\mathcal{D}_p$ above may be re-expressed as (derived) module categories over these annular arc algebras; this will give an annular anlogue of the main result in \cite{khov1}, with the functors categorifying the invariants for the tangles $g_n^i, f_n^i, t_n^i(1), (2)$ and $r_n$ corresponding to tensoring with certain (complexes of) bi-modules. Further, the categories $\mathcal{D}_p$ may be re-expressed as (derived) categories of $\text{Mod}^{fg, \lambda}_{e}(U_{\textbf{k}})$. Then, as in \cite{bfk}, the functors categorifying the invariants for the tangles $g_n^i, f_n^i, t_n^i(1), (2)$ can be expressed using certain equivalences between different blocks of representation categories, due to Enright-Shelton; this will give a positive characteristic analogue of the aforementioned categorification result from \cite{bfk}. While the construction in \cite{bfk} naturally gave rise to Khovanov homology, our construction will give rise to annular Khovanov homology (Grigsby, Licata and Wehrli in \cite{glw}). 

\bibliographystyle{amsalpha}

\vspace{1cm}

\end{document}